\documentclass[reqno,a4paper, 11pt]{amsart}

\usepackage[a4paper=true,pdfpagelabels]{hyperref}
\usepackage{graphicx}

\usepackage[ansinew]{inputenc}
\usepackage{amsfonts,epsfig}
\usepackage{latexsym}
\usepackage{amsmath}
\usepackage{amssymb}
\usepackage{color}

\newtheorem{theorem}{Theorem}
\newtheorem{lemma}[theorem]{Lemma}

\newtheorem{proposition}[theorem]{Proposition}

\newtheorem{lettertheorem}{Theorem}
\newtheorem{letterlemma}[lettertheorem]{Lemma}

\theoremstyle{definition}

\theoremstyle{remark}

\numberwithin{equation}{section}

\def\wo{\widehat{\om}}
\newcommand{\op}{\mathrm{o}}

\newcommand{\D}{\mathbb{D}}
\newcommand{\DD}{\widehat{\mathcal{D}}}
\newcommand{\N}{\mathbb{N}}
\newcommand{\RR}{\mathbb{R}}

\newcommand{\C}{\mathbb{C}}

\newcommand{\e}{\varepsilon}
\newcommand{\ep}{\varepsilon}

\renewcommand{\phi}{\varphi}

\newcommand{\T}{\mathbb{T}}

\newcommand{\wtT}{\widetilde{T}}

\def\a{\alpha}       \def\b{\beta}        \def\g{\gamma}
\def\d{\delta}           \def\e{\varepsilon}
\def\la{\lambda}     \def\om{\omega}      
       \def\t{\theta}       
         \def\r{\rho}         \def\z{\zeta}
\def\F{\Phi}                  \def\vp{\varphi}

\def\R{{\mathcal R}}

\def\Tm{ \mathcal{T}_\mu}

\renewcommand{\H}{\mathcal{H}}
\newcommand{\SSS}{\mathcal{S}}
\newenvironment{Prf}{\noindent{\emph{Proof of}}}
{\hfill$\Box$ }
\begin{document}

\title[ Toeplitz operators and Berezin
transform ]{Berezin transform and  Toeplitz operators on weighted
Bergman spaces induced by regular weights }


\author{Jos\'e \'Angel Pel\'aez}
\address{Jos\'e \'Angel Pel\'aez, Departamento de An\'alisis Matem\'atico, Universidad de M\'alaga, Campus de
Teatinos, 29071 M\'alaga, Spain} \email{japelaez@uma.es}

\author{Jouni R\"atty\"a}
\address{Jouni R\"atty\"a, University of Eastern Finland, P.O.Box 111, 80101 Joensuu, Finland}
\email{jouni.rattya@uef.fi}

\author{Kian Sierra}

\address{Kian Sierra, University of Eastern Finland, P.O.~Box 111, 80101 Joensuu, Finland; Departamento de An\'alisis Matem\'atico, Universidad de M\'alaga, Campus de
Teatinos, 29071 M\'alaga, Spain}

\email{kian.sierra@uef.fi}

\date{\today}

\thanks{This research was supported in part by
by Ministerio de Econom\'{\i}a y Competitivivad, Spain,   projects
MTM2014-52865-P,  and MTM2015-69323-REDT; La Junta de Andaluc{\'i}a,
project FQM210 and Academy of Finland project no. 268009, and the
Faculty of Science and Forestry of University of Eastern Finland
project no. 930349.}

\date{\today}

%%%%%%%%%%%%%%%%%%%%%%%%%%%%%
%%%% ----  ABSTRACT ---- %%%%
%%%%%%%%%%%%%%%%%%%%%%%%%%%%%

\begin{abstract}
Given a regular weight $\omega$ and a positive Borel measure
$\mu$ on the unit disc~$\D$, the Toeplitz operator associated with $\mu$
is
    $$
    \mathcal{T}_\mu(f)(z)=\int_\D f(\z)\overline{B_z^\omega(\z)}\,d\mu(\z),
    $$
where $B^\om_{z}$ are the reproducing kernels of the weighted
Bergman space $A^2_\om$. We describe bounded and
compact Toeplitz operators $\mathcal{T}_\mu:A^p_\omega\to
A^q_\omega$, $1<q,p<\infty$, in terms of Carleson measures and the Berezin transform
    $$
    \widetilde{\mathcal{T}}_\mu(z)=\frac{\langle\mathcal{T}_\mu(B^\om_{z}), B^\om_{z} \rangle_{A^2_\omega}}{\|B_z^\omega\|^2_{A^2_\omega}}.
    $$
We also characterize Schatten class Toeplitz operators in terms of the Berezin transform and apply this result to study Schatten class composition operators.
\end{abstract}

\maketitle

\section{Introduction and main results}

Let $\H(\D)$ denote the space of analytic functions in the unit disc $\D=\{z\in\C:|z|<1\}$. For $0<p<\infty$ and a nonnegative
integrable function $\om$ on $\D$, the weighted Bergman space~$A^p_\omega$ consists of $f\in\H(\D)$ such
that
    $$
    \|f\|_{A^p_\omega}^p=\int_\D|f(z)|^p\omega(z)\,dA(z)<\infty,
    $$
where $dA(z)=\frac{dx\,dy}{\pi}$ is the normalized Lebesgue area
measure on $\D$. As usual, $A^p_\alpha$ denotes the weighted Bergman
space induced by the standard radial weight $(1-|z|^2)^\alpha$.

A radial weight $\om$ belongs to the class~$\DD$ if
$\widehat{\om}(z)=\int_{|z|}^1\om(s)\,ds$ satisfies the doubling condition $\widehat{\om}(r)\le
C\widehat{\om}(\frac{1+r}{2})$. Further, a radial weight $\om\in\DD$ is regular,
denoted by $\om\in\R$, if $\om(r)$ behaves as its
integral average over $(r,1)$, that is,
    \begin{equation*}
    \om(r)\asymp\frac{\int_r^1\om(s)\,ds}{1-r},\quad 0\le r<1.
    \end{equation*}
Every standard weight as well as those given in
\cite[(4.4)--(4.6)]{AS} are regular. It is easy to see that for each radial weight $\om$, the norm convergence in $A^2_\om$ implies
the uniform convergence on compact subsets of $\D$, and hence the
Hilbert space $A^2_\om$ is a closed subspace of $L^2_\om$ and the
orthogonal Bergman projection $P_\om$ from $L^2_\om$ to $A^2_\om$ is
given by
    \begin{equation*}
    P_\om(f)(z)=\int_{\D}f(\z)\overline{B^\om_{z}(\z)}\om(\z)\,dA(\z),
    \end{equation*}
where $B^\om_{z}$ are the reproducing kernels of $A^2_\om$. Recently, those regular weights $\om$ and $\nu$ for which $P_\om:L^p_\nu\to L^p_\nu$ is bounded were characterized in terms of
Bekoll\'e-Bonami type conditions~\cite{Twoweight}. In this paper we consider operators which are natural extensions of  the orthogonal projection $P_\om$.
For a positive Borel measure $\mu$ on $\D$, the Toeplitz operator associated with $\mu$ is
    $$
    \mathcal{T}_\mu(f)(z)=\int_\D f(\z)\overline{B_z^\om(\z)}\,d\mu(\z).
    $$
If $d\mu=\Phi \om dA$ for a non-negative function $\Phi$, then write $\mathcal{T}_\mu=\mathcal{T}_\Phi$ so that $\mathcal{T}_\Phi(f)=P_\om(f\Phi)$.
The operator~$\mathcal{T}_\Phi$ has been extensively studied since the seventies \cite{CobInd73, McSuInd79, ZhuTams87}.
Luecking was probably the one who introduced Toeplitz operators $\mathcal{T}_\mu$
with measures as symbols in~\cite{Lu87}, where he provides, among other things, a
description of Schatten class Toeplitz
operators $\mathcal{T}_\mu:A^2_\a\to A^2_\a$ in terms of an $\ell^p$-condition involving a hyperbolic lattice of $\D$. More recently, Zhu~\cite{ZhuNY07} gave an alternative characterization in terms of $L^p\left(\frac{dA}{(1-|\cdot|)^2}\right)$-integrability
of the Berezin transform of $\mathcal{T}_\mu$ in the widest
possible range of the paremeters $p$ and $\alpha$. We refer to
\cite[Chapter~7]{Zhu} for the theory of Toeplitz operators~$\mathcal{T}_\mu$ acting on $A^2_\alpha$ and to
\cite{ChoeKooYiPotAnal02,PauZhao} for descriptions in terms of
Carleson measures and the Berezin transform of bounded and compact
Toeplitz operators $\mathcal{T}_\mu:A^p_\alpha\to
A^q_\alpha$, $1<p,q<\infty$. The Berezin transform of a bounded
linear operator $T:A^2_\omega \to A^2_\omega$ is
    \begin{equation}\label{101}
    \widetilde{T}(z)=\langle T(b^\om_z), b^\om_z\rangle_{A^2_\om},
    \end{equation}
where $b^\om_z=\frac{B^\om_z}{\| B^\om_z\|_{A^2_\om}}$ are the
normalized reproducing kernels of $A^2_\om$. Given $0<p,q<\infty$
and a positive Borel measure $\mu$ on $\D$, we say that $\mu$ is a
$q$-Carleson measure for $A^p_\om$ if the identity operator
$Id:A^p_\om\to L^q_\mu$ is bounded. A description of $q$-Carleson
measures for $A^p_\om$ induced by doubling weights was recently given in~\cite{PelRatEmb}, see also \cite{PelRatSie2015}.

One of the main purposes of this study is to characterize, in terms of Carleson measures
and the Berezin transform $\widetilde{\mathcal{T}}_\mu$, those
positive Borel measures $\mu$ such that the Toeplitz operator $\mathcal{T}_\mu:
A^p_\om\to A^q_\om$, where $1<p,q<\infty$ and $\om\in\R$, is bounded or compact. We also describe Schatten class Toeplitz operators $\mathcal{T}_\mu:A^2_\om\to A^2_\om$ in terms of their Berezin transforms and show how this result can be used to study Schatten class composition operators induced by symbols of bounded valence.

A simple fact that is repeatedly used in the study of Toeplitz operators on standard
Bergman spaces $A^p_\a$ is the closed formula
$(1-\overline{z}\zeta)^{-(2+\alpha)}$ of the Bergman reproducing
kernel of $A^2_\alpha$. This shows that the kernels never vanish, and allows one to easily establish useful pointwise and norm estimates.
However, the situation in the case of $A^2_\om$ with $\om\in\R$ is more
complicated because of the lack of such an explicit expression for
$B^\om_z$. In fact a little perturbation in the weight, that does
not change the space itself, might introduce zeros to the kernel
functions \cite{ZeyPams10}. This difference causes severe difficulties in the study related to Toeplitz operators on $A^p_\om$, and forces us to circumvent several obstacles in a different manner. We will shortly indicate the main tools used in the proofs after each result is stated.

We need a bit more of notation to state our first result. For each $1<p<\infty$ we write~$p'$ for its conjugate
exponent, that is, $\frac{1}{p}+\frac{1}{p'}=1$. The Carleson square $S(I)$ based on an interval $I$ on the boundary $\T$ of $\D$ is the
set $S(I)=\{re^{it}\in\D:e^{it}\in I,\, 1-|I|\le r<1\}$, where
$|E|$ denotes the Lebesgue measure of $E\subset\T$. We associate to
each $a\in\D\setminus\{0\}$ the interval $I_a=\{e^{i\t}:|\arg(a
e^{-i\t})|\le\frac{1-|a|}{2}\}$, and denote $S(a)=S(I_a)$.

\begin{theorem}\label{th:tubounpq}
Let $1<p\le q<\infty$, $\om\in\R$ and $\mu$ be a positive Borel
measure on $\D$. Then the following statements are equivalent:
    \begin{enumerate}
    \item[(i)] $\mathcal{T}_\mu:A^p_\om\to A^q_\om$ is bounded;
    \item[(ii)] $\frac{\widetilde{\mathcal{T}}_\mu(\cdot)}{\om(S(\cdot))^{\frac{1}{p}+\frac{1}{q'}-1}}\in L^\infty$;
    \item[(iii)] $\mu$ is a $\frac{s(p+q')}{pq'}$-Carleson measure for $A^s_\om$ for some (equivalently for all)
    $0<s<\infty$;
    \item[(iv)]
    $\frac{\mu(S(\cdot))}{\om(S(\cdot))^{\frac1p+\frac1{q'}}}\in L^\infty.$
    \end{enumerate}
Moreover,
    $$
    \left\|\Tm\right\|_{A^p_\omega \to A^q_\omega}
    \asymp\left\|\frac{\widetilde{\mathcal{T}}_\mu(\cdot)}{\om(S(\cdot))^{\frac{1}{p}+\frac{1}{q'}-1}}\right\|_{L^\infty}
    \asymp\left\|Id\right\|^{\frac{s(p+q')}{pq'}}_{A^s_\omega \to L^{\frac{s(p+q')}{pq'}}_\mu}
    \asymp \left\|\frac{\mu(S(\cdot))}{\om(S(\cdot))^{\frac1p+\frac1{q'}}}\right\|_{L^\infty}.
    $$
\end{theorem}

The equivalence between (ii) and (iv) shows that the Berezin transform $\widetilde{\mathcal{T}}_\mu$ behaves asymptotically as the average $\mu(S(\cdot))/\om(S(\cdot))$. By using Fubini's theorem and the reproducing formula
 \begin{equation}\label{repfor}
    L_z(f)=f(z)=\langle f, B^\om_{z}\rangle_{A^2_\om} =\int_{\D} f(\z)\,\overline{B^\om_{z}(\z)}\,\om(\z)\,dA(\z), \quad f\in A^1_\om,
    \end{equation}
we deduce
    \begin{equation}\label{tmu}
    \langle \mathcal{T}_\mu(f), g\rangle_{A^2_\om}=\langle f,
    g\rangle_{L^2_\mu}
    \end{equation}
for each compactly supported positive Borel measure $\mu$ and all $f,g\in A^2_\om$. This identity shows that Carleson measures and Toeplitz operators are intimately connected, and thus the use of Carleson measures in the proof of Theorem~\ref{th:tubounpq} does not come as a surprise. Another key tools in the proof are the $L^p$-estimates of the kernels $B^\om_z$, obtained in~\cite[Theorem~1]{Twoweight}, and
a pointwise estimate for $B^\om_z$ in a sufficiently small
Carleson square contained in $S(z)$, given in Lemma~\ref{le:kersquare} below. We also prove a counterpart of Theorem~\ref{th:tubounpq} for compact Toeplitz operators. This result is stated as Theorem~\ref{th:tucompactpq} and its proof relies, among other things, on the duality relation $(A^p_\om)^\star\simeq A^{p'}_\om$ under the pairing $\langle\cdot,\cdot\rangle_{A^2_\om}$, valid for all $\om\in\R$~\cite[Corollary~7]{Twoweight}.

To describe the positive Borel measures such that
$\mathcal{T}_\mu: A^p_\om\to A^q_\om$ is bounded on the range
$1<q<p<\infty$, we write $\varrho(a,z)=|\vp_a(z)|=\left|\frac{a-z}{1-\overline{a}z}\right|$, for the pseudohyperbolic distance between $z$ and $a$,
and $\Delta(a,r)=\{z:\varrho(a,z)<r\}$ for the pseudohyperbolic disc of center $a\in\D$ and radius $r\in(0,1)$.

\begin{theorem}\label{th:qmenorp}
Let $1<q<p<\infty$, $0<r<1$, $\om\in\R$ and $\mu$ be a positive
Borel measure on $\D$. Then the following statements are equivalent:
    \begin{enumerate}
    \item[\rm(i)] $\mathcal{T}_\mu: A^p_\om\to A^q_\om$ is compact;
    \item[\rm(ii)] $\mathcal{T}_\mu: A^p_\om\to A^q_\om$ is bounded;
    \item[\rm(iii)] $\widehat{\mu}_r(\cdot)=\frac{\mu(\Delta(\cdot,r))}{\omega(\Delta(\cdot,r))}
    \in L^\frac{pq}{p-q}_\omega$;
    \item[\rm(iv)] $\mu$ is a $\left(p+1-\frac{p}{q}\right)$-Carleson
    measure for $A^p_\om$;
    \item[\rm(v)] $Id:A^p_\om\to L^{p+1-\frac{p}{q}}_\mu$ is compact;
    \item[\rm(vi)] $\widetilde{\mathcal{T}}_\mu\in L^\frac{pq}{p-q}_\omega$.
    \end{enumerate}
Moreover,
    $$
    \|\mathcal{T}_\mu\|_{A^p_\om\to A^q_\om}
    \asymp\|\widehat{\mu}_r\|_{L^{\frac{qp}{p-q}}_\om}
    \asymp\|Id\|^{p+1-\frac{p}{q}}_{A^p_\om \to L^{p+1-\frac{p}{q}}_\mu}
    \asymp\|\widetilde{\mathcal{T}}_\mu\|_{L^{\frac{qp}{p-q}}_\om}.
    $$
\end{theorem}

Apart from standard techniques, such as a duality relation for Bergman spaces and the use of Rademacher functions along with Khinchine's inequality, the boundedness of the maximal Bergman projection
    $$
    P^+_\om(f)(z)=\int_{\D}|f(\z)||B^\om_{z}(\z)|\,\om(\z)dA(\z)
    $$
on $L^p_\om$ for $p\in(1,\infty)$ and $\om\in\R$~\cite[Theorem~5]{Twoweight} plays a crucial role in the proof of Theorem~\ref{th:qmenorp}. Another important fact employed is that, even if the kernels may vanish, by Lemma~\ref{b6} for each $\om\in\DD$ they obey the relation $|B^\om_a|\asymp B^\om_a(a)$ on sufficiently small pseudohyperbolic discs centered at $a$. This is used when (iii) is considered, but (iii) involves pseudohyperbolic discs of all sizes, and therefore a suitably chosen covering of~$\D$ will be used to deal with this technical obstacle.

As for the statements of our results on Schatten classes, some notation are in order. The polar rectangle associated with an arc $I\subset\T$
is
    $$
    R(I)=\left\{z\in\D:\,\frac{z}{|z|}\in I,\,\,1-\frac{|I|}{2\pi}\le |z|<1-\frac{|I|}{4\pi}\right\}.
    $$
Write $z_I=(1-|I|/2\pi)\xi$, where $\xi\in\T$ is the midpoint of
$I$. Let $\Upsilon$ denote the family of all dyadic arcs of $\T$.
Every arc $I\in\Upsilon$ is of the form
    $$
    I_{n,k}=\left\{e^{i\theta}:\,\frac{2\pi k}{2^n}\le
    \theta<\frac{2\pi(k+1)}{2^n}\right\},
    $$
where $k=0,1,2,\dots,2^n-1$ and $n=\N\cup\{0\}$. The family
$\left\{R(I):\,\,I\in\Upsilon\right\}$ consists of pairwise disjoint
rectangles whose union covers~$\D$. For
$I_j\in\Upsilon\setminus\{I_{0,0}\}$, we will write $z_j=z_{I_{j}}$.
For convenience, we associate the arc $I_{0,0}$ with the point
$1/2$. Given a radial weight $\om$, we write
    $$
    \omega^\star(z)=\int_{|z|}^1\omega(s)\log\frac{s}{|z|}s\,ds,\quad z\in\D\setminus\{0\}.
    $$

\begin{theorem}\label{MTMSchatten}
Let $0<p<\infty$, $\om\in\DD$ and $\mu$ be a positive Borel measure
on $\D$. Then the following statements are equivalent:
\begin{enumerate}
\item[\rm(i)] $\mathcal T_\mu$ belongs to the Schatten $p$-class $\SSS_p(A^2_\om)$;
\item[\rm(ii)]  $\sum_{R_j\in\Upsilon}
    \left(\frac{\mu(R_j)}{\omega^\star(z_j)}\right)^p<\infty$;
\item[\rm(iii)] $\frac{\mu\left(\Delta(\cdot,r)\right)}{\om^\star(\cdot)}$ belongs to $L^p\left(\frac{dA}{(1-|\cdot|)^2}\right)$ for some (equivalently for all) $0<r<1$.
\end{enumerate}
Moreover,
    $$
    |\mathcal{T}_\mu|_p^p\asymp\sum_{R_j\in\Upsilon}
    \left(\frac{\mu(R_j)}{\omega^\star(z_j)}\right)^p
    \asymp\int_\D\left(\frac{\mu\left(\Delta(z,r)\right)}{\om^\star(z)}\right)^p\frac{dA(z)}{(1-|z|)^2}.
    $$
If $\om\in\R$ such that
$\frac{(\om^\star(\cdot))^p}{(1-|\cdot|)^2}$ is also a regular weight, then
$\mathcal T_\mu\in\SSS_p(A^2_\om)$ if and only if
$\widetilde{\mathcal{T}}_\mu\in L^p_{\om/\om^\star}$, and $|\mathcal{T}_\mu|_p\asymp\|\widetilde{\mathcal{T}}_\mu\|_{L^p_{\om/\om^\star}}$.
\end{theorem}

The equivalence of the first three statements
were proved in \cite[Theorem~1]{PR2016/2}, and hence the novelty of Theorem~\ref{MTMSchatten} stems from the last part involving the
Berezin transform. The hypothesis $\frac{\omega^\star(\cdot)^p}{(1-|\cdot|)^2}\in\R$ is not a restriction
for $p\ge 1$, and for $\om(z)=(1-|z|^2)^\alpha$ it reduces to the inequality
$p(2+\a)>1$. Therefore Theorem~\ref{MTMSchatten} is an extension of
\cite[Theorem~7.18]{Zhu}, see also \cite{ZhuNY07}. Since each
standard weight is regular, the cut-off condition
$\frac{\om^\star(\cdot)^p}{(1-|\cdot|)^2}\in\R$ is in a sense the best possible.

The proof of the last statement of Theorem~\ref{MTMSchatten} for
$p\ge 1$ follows by standard techniques once the pointwise kernel estimate given in Lemma~\ref{b6} is available. However, the proof for $0<p<1$
is more involved because the reproducing kernels of $A^2_\om$ with $\om\in\R$ do not necessarily remain essentially constant in hyperbolically
bounded regions, a property which the standard kernels
$(1-\overline{z}\zeta)^{2+\alpha}$ trivially admit and is used in the proof of
\cite[Theorem~7.18]{Zhu} concerning the weighted Bergman spaces $A^p_\alpha$. This obstacle is circumvent by using
subharmonicity and estimates for the $A^p_\nu$-norm of $B^\om_z$ for doubling weights
$\om,\nu\in\DD$, obtained in \cite[Theorem~1]{Twoweight}.

Theorem~\ref{MTMSchatten} can be applied to study Schatten class composition operators when the inducing symbol $\vp$ is of finite valence. To state the result, some more notation and motivation are in order. For an analytic self-map $\vp$ of $\D$, let $\z\in\vp^{-1}(z)$ denote the set of the points $\{\z_n\}$ in
$\D$, organized by increasing moduli and repeated according to their multiplicities, such that $\vp(\z_n)=z$ for
all $n$.
For a radial weight $\om$ and $\vp$ as above, the generalized Nevanlinna counting function is
    $$
    N_{\vp,\om^\star}(z)=\sum_{\z\in\vp^{-1}(z),}\om^\star\left(\z\right),\quad z\in\D\setminus\{\vp(0)\}.
    $$
In \cite[Theorem~3]{PR2016/2} it was shown that, for each $\om\in\DD$, the composition operator $C_\vp$ belongs to the Schatten $p$-class $\SSS_p(A^2_\omega)$ if and only if $N_{\vp,\omega^\star}\in L^p\left( \frac{dA}{(1-|\cdot|)^2 }\right)$. This condition might be difficult to test in praxis because of the
counting function $N_{\vp,\omega^\star}$. Therefore it is natural to look
for more workable descriptions. As for this, we observe that by using \cite[Theorem~1]{Twoweight} one can show that the Berezin transform
of $C_\vp C_\vp^\star$ behaves asymptotically as
$\frac{\omega^\star(\cdot)}{\omega^\star(\vp(\cdot))}$, and moreover, the condition $\frac{\omega^\star(z)}{\omega^\star(\vp(z))}\to0$, $|z|\to1^-$,
characterizes compact operators $C_\vp:A^2_\om\to A^2_\om$ when $\om\in\R$ by
\cite[Theorem~20 and Lemma~23]{PR2016/2}. Therefore one may ask
how close is the condition
    \begin{equation}\label{39intro}
    \int_\D\left(\frac{\omega^\star(z)}{\omega^\star(\vp(z))}\right)^\frac{p}{2}\frac{\omega(z)}{\omega^\star(z)}\,dA(z)<\infty
    \end{equation}
to describe Schatten class composition operators? The next result
shows that this is a description in the case $p>2$ under the hypothesis of $\vp$ being of bounded
valence.

\begin{theorem}\label{Thm:CompositionBoundedValence}
Let $2<p<\infty$ and $\omega\in\R$, and let $\vp$ be a bounded
valent analytic self-map of $\D$. Then $C_\vp\in\SSS_p(A^2_\omega)$
if and only if \eqref{39intro} holds.
\end{theorem}

Theorem~\ref{Thm:CompositionBoundedValence} is an extension of
\cite[Theorem~1.1]{Zhu2001} to the setting of regular weights. If
$\om(z)=(1-|z|^2)^\a$, then the statement in
Theorem~\ref{Thm:CompositionBoundedValence} is not valid for $p(\alpha+2)\le 2$ because in this case the condition \eqref{39intro} fails for all analytic self-maps $\vp$. More generally, by using \cite[p. 10 (ii)]{PelRat} one can show that if $\om\in\R$ and $p$ is small enough, then \eqref{39intro} fails for each $\vp$.
Moreover, \cite[Theorem~3]{XiaPams2003} shows that the
statement in Theorem~\ref{Thm:CompositionBoundedValence} does not remain valid for $\om\equiv1$ without the additional hypothesis regarding the valence of $\vp$.

It is easy to see that each regular weight $\om$ satisfies $\om(r)\asymp\om(t)$ whenever $1-r\asymp1-t$. This asymptotic relation shows that $\om\in\R$ must be essentially constant in each hyperbolically bounded region, and hence, in particular, $\om$ may not have zeros. This apparently severe requirement does not cause too much loss of generality in our study. This because in the next section we will show that if $\om\in\DD$ satisfies the reverse doubling property $\widehat{\om}(r)\ge C\widehat{\om}\left(1-\frac{1-r}{K}\right)$ for some $K>1$ and $C>1$, a condition that is satisfied for each $\om\in\R$, then there exists a differentiable strictly positive weight $W\in\R$ such that $\|\cdot\|_{A^p_\om}$ and $\|\cdot\|_{A^p_W}$ are comparable. In Section~\ref{sec2} we also discuss the kernel estimates and other auxiliary results. Section~\ref{sec3} is devoted to the study of bounded and compact Toeplitz operators. Schatten class Toeplitz and composition operators are discussed in Sections \ref{sec4} and \ref{sec:composition}, respectively.

\section{Pointwise and norm estimates of Bergman reproducing
kernels}\label{sec2}

We begin with considering the classes of weights appearing in this study and their basic properties. Then we will prove several pointwise and norm estimates for the reproducing kernels, and finally an auxiliary result on weak convergence of normalized kernels is established.

The first auxiliary lemma contains several characterizations of doubling weights and will be repeatedly used throughout the rest of the paper. For a proof, see \cite[Lemma~2.1]{PelSum14}. All along we will assume $\widehat{\om}(r)>0$ for all $0\le r<1$ without mentioning it, for otherwise $A^p_\om=\H(\D)$.

\begin{letterlemma}\label{Lemma:replacement-Lemmas-Memoirs}
Let $\om$ be a radial weight. Then the following conditions are
equivalent:
\begin{itemize}
\item[\rm(i)] $\om\in\DD$;
\item[\rm(ii)] There exist $C=C(\om)>0$ and $\b=\b(\om)>0$ such that
    \begin{equation*}
    \begin{split}
    \widehat{\om}(r)\le C\left(\frac{1-r}{1-t}\right)^{\b}\widehat{\om}(t),\quad 0\le r\le t<1;
    \end{split}
    \end{equation*}
\item[\rm(iii)] There exist $C=C(\om)>0$ and $\gamma=\gamma(\om)>0$ such that
    \begin{equation*}
    \begin{split}
    \int_0^t\left(\frac{1-t}{1-s}\right)^\g\om(s)\,ds
    \le C\widehat{\om}(t),\quad 0\le t<1;
    \end{split}
    \end{equation*}
\item[\rm(iv)] The asymptotic equality
    $$
    \int_0^1s^x\om(s)\,ds\asymp\widehat{\om}\left(1-\frac1x\right),\quad x\in[1,\infty),
    $$
is valid;
\item[\rm(v)] $\om^\star(z)\asymp\widehat{\om}(z)(1-|z|)$, $|z|\to1^-$;
\item[\rm(vi)] There exists $\lambda=\lambda(\om)\ge0$ such that
    $$
    \int_\D\frac{\om(z)}{|1-\overline{\z}z|^{\lambda+1}}dA(z)\asymp\frac{\widehat{\om}(\zeta)}{(1-|\z|)^\lambda},\quad \z\in\D;
    $$
\item[\rm(vii)] There exists $C=C(\om)>0$ such that the moments  $\om_n=\int_0^1 r^{n}\om(r)\,dr$ satisfy the condition
    $\om_{n}\le C\om_{2n}$.
\end{itemize}
\end{letterlemma}

We next briefly discuss radial weights having a kind of reversed doubling property, and then show how this is related to the pointwise condition that defines the class $\R$ of regular weights. More precisely, we show that if $\om\in\DD$ satisfies the reverse doubling condition appearing in part (i) of Lemma~\ref{Lemma:weights-in-R} below, then one can find a strictly positive $n$ times differentiable weight which belongs to $\R$ and induces the same Bergman space as $\om$. The next lemma can be find in \cite{TwoweightII}.

\begin{letterlemma}\label{Lemma:weights-in-R}
Let $\om$ be a radial weight. For each $K>1$, let $\r_n=\r_n(\om,K)$ be the sequence defined by $\widehat{\om}(\r_n)=\widehat{\om}(0)K^{-n}$, and for each $\b\in\RR$, write $\om_{[\b]}(z)=\om(z)(1-|z|)^\b$. Then the following statements are equivalent:
\begin{itemize}
\item[\rm(i)] There exist $K=K(\om)>1$ and $C=C(\om)>1$ such that $\widehat{\om}(r)\ge C\widehat{\om}\left(1-\frac{1-r}{K}\right)$ for all $0\le r<1$;
\item[\rm(ii)] There exist $C=C(\om)>0$ and $\b=\b(\om)>0$ such that
    \begin{equation*}
    \begin{split}
    \widehat{\om}(t)\le C\left(\frac{1-t}{1-r}\right)^{\b}\widehat{\om}(r),\quad 0\le r\le t<1;
    \end{split}
    \end{equation*}
\item[\rm(iii)] For some (equivalently for each) $\b\in(0,\infty)$, there exists $C=C(\b,\om)\in(0,1)$ such that
    $$
    \frac{\int_r^1\widehat{\om}(t)\b(1-t)^{\b-1}\,dt}{(1-r)^{\b}}\le C\widehat{\om}(r),\quad 0<r<1.
    $$
\end{itemize}
\end{letterlemma}

By Lemma~\ref{Lemma:weights-in-R} and \cite[Lemma~1.1]{PelRat} each $\om\in\R$ satisfies the reverse doubling condition.The next result shows that if $\om\in\DD$ satisfies
the reverse doubling condition, then there exists a continuous and locally smooth weight~$W$ that induces the same Bergman space as $\om$.

\begin{proposition}\label{Lemma:w-continuous}
Let $0<p<\infty$ and $\om\in\DD$, and write
$W(r)=W_\om(r)=\widehat{\om}(r)/(1-r)$ for all $0\le r<1$. Then
$\|f\|_{A^p_W}\asymp\|f\|_{A^p_\om}$ for all $f\in\H(\D)$ if and
only if $\om$ satisfies the reverse doubling condition appearing in part (i) of Lemma~\ref{Lemma:weights-in-R}.
\end{proposition}

\begin{proof}
Since $\om$ belongs to $\DD$ by the hypothesis, so does $W$. Therefore $\|f\|_{A^p_W}\asymp\|f\|_{A^p_\om}$ for all $f\in\H(\D)$ by \cite[Theorem~1]{PelRatEmb} if $W(S(a))\asymp\om(S(a))$ for all $a\in\D\setminus\{0\}$. Since $\om$ and $W$ are radial, this is the case if
    $$
    \widehat{W}(r)%=\int_r^1\frac{\widehat{\om}(t)}{1-t}\,dt
    =\widehat{\om}(r)\int_r^1\frac{\widehat{\om}(t)}{\widehat{\om}(r)}\frac{1}{1-t}\,dt\asymp\widehat{\om}(r),\quad 0\le r<1.
    $$
If now $\om\in\DD$ satisfies the reverse doubling condition, then Lemma~\ref{Lemma:replacement-Lemmas-Memoirs}(ii) and Lemma~\ref{Lemma:weights-in-R}(ii) applied to the middle term above imply the asymptotic equality we are after.

Conversely, assume that $\om\in\DD$ and $\|f\|_{A^p_W}\asymp\|f\|_{A^p_\om}$ for all $f\in\H(\D)$. Write $f_a(z)=(1-\overline{a}z)^{-\frac{\lambda+1}{p}}$ for all $a\in\D$. By Lemma~\ref{Lemma:replacement-Lemmas-Memoirs}(vi) there exists $\lambda=\lambda(\om)\ge0$ such that
    \begin{equation*}
    \begin{split}
    \frac{\widehat{\om}(a)}{(1-|a|)^\lambda}
    &\asymp\int_\D\frac{\om(z)}{|1-\overline{a}z|^{\lambda+1}}\,dA(z)=\|f_a\|_{A^p_\om}^p\asymp\|f_a\|_{A^p_W}^p\asymp
    \int_0^1\frac{\widehat{\om}(r)}{(1-|a|r)^{\lambda}(1-r)}\,dr\\
    &\ge\int_{|a|}^1\frac{\widehat{\om}(r)}{(1-|a|r)^{\lambda}(1-r)}\,dr
    \gtrsim\frac{\widehat{\widehat{\om}}(r)}{(1-|a|)^{\lambda+1}},
    \end{split}
    \end{equation*}
and thus $\om$ satisfies the Lemma~\ref{Lemma:weights-in-R}(iii) with $\b=1$.
\end{proof}

Consider now $\om\in\DD$ satisfying the reverse doubling condition. Then $A^p_\om=A^p_{W_\om}$ and $W_\om\in\R$ by the first part of the proof of Proposition~\ref{Lemma:w-continuous}. The weight $W_\om$ is continuous and strictly positive. Further, the differentiable weight $\widehat{W_\om}(r)/(1-r)$ belongs to $\R$ and induces the same Bergman space as $\om$. Therefore, by repeating the process, for a given $\om\in\DD$ satisfying the reverse doubling condition, we can always find a strictly positive $n$ times differentiable weight that induces the same Bergman space as the original weight~$\om$. Therefore assuming $\om\in\R$ instead of the two doubling conditions is not a severe restriction in our study.

The true advantage of the class $\R$ is the local smoothness of its weights. It is clear that if $\om\in\R$, then for each $s\in[0,1)$
there exists a constant $C=C(s,\omega)>1$ such that
    \begin{equation}\label{eq:r2}
    C^{-1}\om(t)\le \om(r)\le C\om(t),\quad 0\le r\le t\le
    r+s(1-r)<1.
    \end{equation}
Therefore, for $\om\in\R$ and $r\in(0,1)$,
    \begin{equation}\label{eq:discocaja}
    \om\left(S(z)\right)\asymp\widehat{\om}(z)(1-|z|)\asymp\om(z)(1-|z|)^2
    \asymp\om\left(\Delta(z,r)\right),\quad z\in\D,
    \end{equation}
where the constants of comparison depend on $\om$ and also on $r$ in the last case. This observation finishes our discussion on basic properties of different classes of weights.

We next turn to kernel estimates. In order to prove our main results, and in particular to deal with the Berezin transform of a Toeplitz operator,
we will need asymptotic estimates for the norm of the Bergman reproducing kernel
in several spaces of analytic functions in $\D$. The next
result follows by \cite[Theorem~1]{Twoweight} (see
also \cite[Lemma~6.2]{PelRat}), Lemma~\ref{Lemma:replacement-Lemmas-Memoirs} and \eqref{eq:discocaja}.

\begin{lettertheorem}\label{Lemma:kernels}
Let $\om,\nu\in\DD$, $0<p<\infty$ and $n\in\N\cup\{0\}$. Then
    \begin{equation}\label{kernel-estimate}
    \|(B^\om_z)^{(n)}\|_{A^p_\nu}^p\asymp\int_0^{|z|}\frac{\widehat{\nu}(t)}{\widehat{\om}(t)^p(1-t)^{p(n+1)}}\,dt,\quad|z|\to1^-.
    \end{equation}
In particular, if $1<p<\infty$, $\om\in\R$ and $r\in(0,1)$, then
    \begin{equation}\label{kernel-estimate-R}
    \|B^\om_z\|_{A^p_\om}^p\asymp\frac{1}{\om(S(z))^{p-1}}\asymp\frac{1}{\om(\Delta(z,r))^{p-1}},\quad
    z\in\D.
    \end{equation}
\end{lettertheorem}

As usual, we write $H^\infty$ for the space of bounded analytic functions in $\D$, and
$\mathcal{B}$ stands for the Bloch functions, that is, the space of $f\in\H(\D)$ such that
$\|f\|_{\mathcal{B}}=\sup_{z\in\D}|f'(z)|(1-|z|)+|f(0)|<\infty$.

\begin{lemma}\label{Blochnorm}
Let $\om\in\DD$. Then
    $$
    \| B^\om_z\|_{\mathcal{B}}\asymp \frac{1}{\omega(S(z))} \asymp \| B^\om_z\|_{H^\infty}, \quad z\in\D.
    $$
\end{lemma}

\begin{proof}
Since
    $$
    B^\om_z(\z)=\sum^\infty_{n=0}\frac{(\z \overline{z})^n}{2 \om_n},
    \quad(B^\om_z)'(\z)=\sum^\infty_{n=1}\frac{n\z^{n-1}
    \overline{z}^n}{2 \om_n},\quad z,\z\in\D,
    $$
the estimate \cite[(20)]{Twoweight}, with $p=1$, $N=2$ and $r=|z|^2$, together with Lemma~\ref{Lemma:replacement-Lemmas-Memoirs} yields
    \begin{equation}\label{kusiputki}
    \begin{split}
    &\left|(B^\om_z)'(z)\right|
    \asymp\sum^\infty_{n=1}\frac{n|z|^{2(n-1)}}{\om_n}
    \asymp \int_0^{|z|^2}\frac{1}{\widehat{\om}(t)(1-t)^3}\,dt\\
    &\asymp\frac{1}{\widehat{\om}(z^2)(1-|z|^2)^2}
    \asymp \frac{1}{\om(S(z))(1-|z|)},\quad |z|\to 1^-,
    \end{split}
    \end{equation}
and hence
    $$
    \frac{1}{\omega(S(z))}\lesssim\|B^\om_z\|_{\mathcal{B}}, \quad |z|\to 1^-.
    $$
Since $\|B^\om_z\|_{\mathcal{B}}\le2\left\|B^\omega_z\right\|_{H^\infty}$, it remains to establish the desired upper estimate for the $H^\infty$-norm.  To see this, observe first that
    \begin{equation*}
    \begin{split}
    \left|B^\om_z(\z) \right|\le\sum^\infty_{n=0}\frac{|z|^n}{2\om_n},\quad z,\z\in\D.
    \end{split}
    \end{equation*}
Then, by using again the estimate \cite[(20)]{Twoweight}, but now with $p=1$,
$N=1$ and $r=|z|$, it follows that
    $$
    \left\| B^\omega_z \right\|_{H^\infty}\le \sum^\infty_{n=0}\frac{ |z|^n}{2 \om_n}
    \asymp \int_0^{|z|}\frac{dt}{\widehat{\omega}(t)(1-t)^2} \asymp \frac{1}{\omega(S(z))},\quad |z|\to 1^-.
    $$
This finishes the proof.
\end{proof}

We next establish two local pointwise estimates for the Bergman
reproducing kernels. To do this, for each
$\delta\in(0,1]$ and $a\in\D\setminus\{0\}$, write
$a_\delta=(1-\delta(1-|a|))e^{i\arg a}$. Then $a_1=a$, $|a_\d|>|a|$ for all $\d\in(0,1)$, and $\lim_{\d\to0^+}a_\d=a/|a|$.

\begin{lemma}\label{le:kersquare} Let $\om\in\DD$. Then there exist constants $c=c(\om)>0$ and $\delta=\delta(\om)\in(0,1]$
such that
    \begin{equation}\label{ke0}
    |B^\om_a(z)|\ge\frac{c}{\om(S(a))},\quad z\in S(a_\delta),\quad a\in\D\setminus\{0\}.
    \end{equation}
\end{lemma}

\begin{proof}
By Theorem~\ref{Lemma:kernels} there exists a constant
$C_1=C_1(\om)>0$ such that $\|B_a^\om\|_{A^2_\om}^2\ge
C_1/\om(S(a))$ for all $a\in\D\setminus\{0\}$, and hence
    \begin{equation}
    \begin{split}\label{ke1}
    |B^\om_a(z)|
    &\ge|B^\om_a(a_\delta)|-|B^\om_a(a_\delta)-B^\om_a(z)|
    =|B^\om_{\sqrt{|aa_\delta|}}(\sqrt{|aa_\delta|})|-|B^\om_a(a_\delta)-B^\om_a(z)|\\
    &=\left\|B^\om_{\sqrt{|aa_\delta|}}\right\|^2_{A^2_\om}-|B^\om_a(a_\delta)-B^\om_a(z)|
    \ge\frac{C_1}{\om(S(\sqrt{|aa_\delta|})}-|B^\om_a(a_\delta)-B^\om_a(z)|\\
    &\ge\frac{C_1}{\om(S(a))}-|B^\om_a(a_\delta)-B^\om_a(z)|,\quad z\in\D.
    \end{split}
    \end{equation}
Moreover, by \eqref{kusiputki} and Lemma~\ref{Lemma:replacement-Lemmas-Memoirs},
    \begin{equation*}
    \begin{split}
    |B^\om_a(a_\delta)-B^\om_a(z)|
    &\le\sup_{\z\in{[a_\delta,z]}}|(B^\om_a)'(\z)||z-a_\d|
    \le2\delta(1-|a|)\sup_{\z\in{[a_\delta,z]}}|(B^\om_a)'(\z)|\\
    &\le\delta(1-|a|)\sum_{n=1}^\infty\frac{n|a|^n}{\om_{2n+1}}
    \le\frac{\delta C_2}{\om(S(a))}.
    \end{split}
    \end{equation*}
By combining this with \eqref{ke1}, and choosing $\d=C_1/2C_2$ we
deduce the assertion for $c=C_1/2$.
\end{proof}

\begin{lemma}\label{b6}
Let $\om\in\DD$. Then there exists $r=r(\om)\in(0,1)$
such that $|B_a^\om(z)|\asymp B_a^\om(a)$ for all $a\in\D$ and $z\in\Delta(a,r)$.
\end{lemma}

\begin{proof}
The proof is similar to that of \cite[Lemma~6.4]{PelRat}. First, use
the Cauchy-Schwarz inequality, Theorem~\ref{Lemma:kernels} and
Lemma~\ref{Lemma:replacement-Lemmas-Memoirs} to obtain
    \begin{equation}\label{102}
    \begin{split}
    |B_a^\om(z)|
    &\le\sum_n\frac{|az|^n}{2\om_{2n+1}}
    \le\left(\sum_n\frac{|z|^{2n}}{2\om_{2n+1}}\right)^\frac12
    \left(\sum_n\frac{|a|^{2n}}{2\om_{2n+1}}\right)^\frac12
    =|B_a^\om(a)|^\frac12|B_z^\om(z)|^\frac12\\
    &\asymp\frac{|B_a^\om(a)|^\frac12}{\sqrt{\widehat{\om}(z)(1-|z|)}}
    \asymp\frac{|B_a^\om(a)|^\frac12}{\sqrt{\widehat{\om}(a)(1-|a|)}}\asymp|B_a^\om(a)|,\quad z\in\Delta(a,r),
    \end{split}
    \end{equation}
for all $a\in\D$. This gives the claimed upper bound. To obtain the same lower bound, let
$r\in(0,1)$ and note first that
    \begin{equation*}
    \begin{split}
    |B_a^\om(z)|&\ge|B_a^\om(a)|-\max_{\z\in[a,z]}|(B_a^\om)'(\z)||z-a|\\
    &\ge|B_a^\om(a)|-\max_{\z\in[a,z]}|(B_a^\om)'(\z)|rC(1-|a|),
    \end{split}
    \end{equation*}
where $C=C(r)>0$ is a constant for which $\sup_{0<r<r_0}C(r)<\infty$ for each $r_0\in(0,1)$.
Now the Cauchy integral formula and a reasoning similar to that in \eqref{102} yield
    $$
    \max_{\z\in[a,z]}|(B_a^\om)'(\z)|\lesssim\frac{|B_a^\om(a)|}{1-|a|},\quad a\in\D,
    $$
and the desired lower bound follows by choosing $r$ sufficiently small.
\end{proof}

The last aim of this section is to show that for each $\om\in\R$, the normalized reproducing kernels $b^\om_{p,z}=B^\om_{z}/\|
B^\om_z\|_{A^p_\om}$ converge weakly to zero in $A^p_\om$, as
$|z|\to 1^-$. To do this, the following growth estimate is used.

\begin{lemma}\label{le:2}
Let $0<p<\infty$ and $\om\in\DD$. Then
    $$
    |f(z)|=\op\left(\frac{1}{\left(\wo(z)(1-|z|)\right)^{\frac{1}{p}}}\right),\quad |z|\to 1^-,
    $$
for all $f\in A^p_\om$.
\end{lemma}

\begin{proof}
Let $f\in A^p_\om$ and $\ep>0$. Then there exists $r\in(0,1)$ such that
    $$
    \ep>\int_{r}^1 M_p^p(s,f)s\om(s)\,ds\ge M_p^p(r,f)r\wo(r),
    $$
which together with the well-known estimate
    $$
    M_\infty(r,f)\lesssim\frac{M_p\left(\frac{1+r}{2},f\right)}{(1-r)^{\frac1p}},\quad 0<r<1,
    $$
and the hypothesis $\om\in\DD$ yields the assertion.
\end{proof}

The proof of the weak convergence we are after relies on the following known duality relation \cite[Corollary~7]{Twoweight}.

\begin{lettertheorem}\label{th:duality}
Let $1<p<\infty$ and $\om\in\R$. Then $(A^p_\om)^\star\simeq A^{p'}_\om$, with equivalence of norms, under the pairing
    \begin{equation}\label{pairing}
    \langle f,g\rangle_{A^2_\om}=\int_{\D}f(z)\overline{g(z)}\om(z)\,dA(z).
    \end{equation}
\end{lettertheorem}

With these preparations we can prove the last result of the section.

\begin{lemma}\label{le:3}
Let $1<p<\infty$ and $\om\in\R$. Then $b^\om_{p,z}\to0$ weakly in
$A^p_\om$, as $|z|\to 1^-$.
\end{lemma}

\begin{proof}
Let $1<p<\infty$ and $\om\in\R$. By Theorem~\ref{th:duality} it suffices to show that
    $$
    \left|\left\langle b^\om_{p,z},g\right\rangle_{A^2_\om}\right|=\frac{|g(z)|}{\|B^\om_z\|_{A^p_\om}}
    \to 0,\quad|z|\to 1^-,
    $$
for all $g\in A^{p'}_\om$. But since
$\|B^\om_z\|_{A^p_\om}^p\asymp\left(\wo(z)(1-|z|)\right)^{1-p}$ by
Theorem~\ref{Lemma:kernels}, and $1-p=-p/p'$, the assertion follows
by Lemma~\ref{le:2}.
\end{proof}

\section{Bounded and compact Toeplitz operators}\label{sec3}

The main objective of this section is to prove
Theorems~\ref{th:tubounpq} and~\ref{th:qmenorp}, stated in the introduction, and establish a characterization analogous to Theorem~\ref{th:tubounpq} for compact operators $\mathcal{T}_\mu:A^p_\om\to
A^q_\om$, given as Theorem~\ref{th:tucompactpq} below.
We begin with the following technical result.

\begin{lemma}\label{fubini}
Let $\mu$ be a finite positive Borel measure on $\D$. Then \eqref{tmu} is satisfied for all $f(z)=\sum_{n=0}^\infty\widehat
{f}(n)z^n$ and $g(z)=\sum_{n=0}^\infty\widehat {g}(n)z^n$ such that
$f\in H^\infty$ and $\sum_{n=0}^\infty|\widehat {g}(n)|<\infty$.
\end{lemma}

\begin{proof}
Fubini's theorem and the dominated convergence theorem yield
    \begin{equation*}
    \begin{split}
    \langle\mathcal T_\mu(f),g\rangle
    &=\lim_{s\to 1^-}\int_{|u|<s}\left(\int_{\D}f(\z)B^\om_{\z}(u)\,d\mu(\z)\right)\overline{g(u)}\om(u)\,dA(u)\\
    &=\lim_{s\to 1^-}\int_{\D}f(\z)\left(\int_{|u|<s}\overline{g(u)}B^\om_\z(u)\om(u)\,dA(u)\right)d\mu(\z)\\
    &=\lim_{s\to 1^-}\int_{\D}f(\z)\left(\sum_{n=0}^\infty\frac{\overline{\widehat{g}(n)\z^n}\int_0^sx^{2n+1}\om(x)\,dx}{\om_{2n+1}}\right)d\mu(\z)
    =\int_{\D}f(\z)\overline{g(\z)}\,d\mu(\z),
    \end{split}
    \end{equation*}
and the assertion is proved.
\end{proof}

Recall that $b^\om_{z}=B^\om_{z}/\|
B^\om_z\|_{A^2_\om}$ for all $z\in\D$. If $\mu$ is a finite positive Borel measure on~$\D$ and $\om\in\DD$,
then by using the definition \eqref{101} of Berezin transform, Lemma~\ref{fubini} and Theorem~\ref{Lemma:kernels}, we deduce
    \begin{equation}\label{Eq:BerezinTransformation}
    \widetilde{\mathcal{T}}_\mu(z)=\langle \mathcal{T}_\mu(b^\om_z), b^\om_z\rangle_{A^2_\om}
    =\frac{\|B^\om_z\|^2_{L^2_\mu}}{\|B^\om_z\|^2_{A^2_\om}}
    \asymp\om\left(S(z)\right)\|B^\om_z\|^2_{L^2_\mu},\quad z\in\D.
    \end{equation}

We now embark on the proofs by considering the cases $p\le q$ and $p>q$ separately.

\subsection{Case $1<p\le q<\infty$}

We first consider bounded Toeplitz operators.

\medskip

\begin{Prf}{\em{Theorem~\ref{th:tubounpq}}.} Since $\frac{p+q'}{pq'}\ge1$ by the hypothesis $q\ge p$, the equivalence (iii)$\Leftrightarrow$(iv) and the estimate
    $$
    \left\|Id\right\|^{\frac{s(p+q')}{pq'}}_{A^s_\omega \to
L^{\frac{s(p+q')}{pq'}}_\mu}
    \asymp \sup_{I\subset\T}\frac{\mu(S(I))}{\om(S(I))^{\frac1p+\frac1{q'}}}
    $$
follow by \cite[Theorem~1]{PelRatEmb}, see also \cite[Theorem~3]{PelRatSie2015} and \cite[Theorem~2.1]{PelRat}.

If $\mathcal{T}_\mu:A^p_\om\to A^q_\om$ is bounded, then H\"older's
inequality and Theorem~\ref{Lemma:kernels} yield
    \begin{equation*}
    \begin{split}
    \left|\widetilde{\mathcal{T}}_\mu(z)\right|
    &=\left|\langle \mathcal{T}_\mu(b^\om_z), b^\om_z\rangle_{A^2_\om}\right|
    \le\left\|\Tm(b^\om_z)\right\|_{A^q_\om}\left\|b^\om_z\right\|_{A^{q'}_\om}
    \le\left\|\Tm\right\|_{A^p_\om \to A^q_\om}\left\|b^\om_z\right\|_{A^p_\om}\left\|b^\om_z\right\|_{A^{q'}_\om}\\
    &=\left\|\Tm\right\|_{A^p_\om \to A^q_\om}\frac{\left\|B^\om_z\right\|_{A^p_\om}
    \left\|B^\om_z\right\|_{A^{q'}_\om}}{\left\|B^\om_z\right\|^2_{A^2_\om}}
    \asymp\left\|\Tm\right\|_{A^p_\om \to A^q_\om}\frac{\omega(S(z))}{\omega(S(z))^{1-\frac{1}{p}}\omega(S(z))^{1-\frac{1}{q'}}}\\
    &\lesssim\left\|\Tm\right\|_{A^p_\om \to A^q_\om} \frac{1}{\omega(S(z))^{1-\frac{1}{p}-\frac{1}{q'}}},\quad
    z\in\D,
    \end{split}
    \end{equation*}
and hence
$\left\|\frac{\widetilde{\mathcal{T}}_\mu(\cdot)}{\om(S(\cdot))^{\frac{1}{p}+\frac{1}{q'}-1}}\right\|_{L^\infty}\lesssim\left\|\Tm\right\|_{A^p_\om\to A^q_\om}$.

Assume next $\frac{\widetilde{\mathcal{T}}_\mu(\cdot)}{\om(S(\cdot))^{\frac{1}{p}+\frac{1}{q'}-1}}\in
L^\infty$, and let $\delta=\delta(\om)$ and $c=c(\om)$ be those of
Lemma~\ref{le:kersquare}. Then Theorem~\ref{Lemma:kernels} and
\eqref{Eq:BerezinTransformation} give
    \begin{equation*}
    \begin{split}
    \frac{\widetilde{\mathcal T}_\mu(z)}{\om(S(z))}
    &\asymp\|B^\om_z\|^2_{A^2_\om}\widetilde{\mathcal T}_\mu(z)
    =\|B^\om_z\|^2_{L^2_\mu}\\
    &\ge\int_{S(z_{\delta})}|B^\om_z(\z)|^2\,d\mu(\z)
    \ge c^2\frac{\mu(S(z_{\delta}))}{\om(S(z))^2},\quad z\in\D\setminus\{0\},
    \end{split}
    \end{equation*}
and hence $\mu(S(z_{\delta}))\lesssim\widetilde{\mathcal T}_\mu(z)\om(S(z))$ for all $z\in\D\setminus\{0\}$. It follows from Lemma~\ref{Lemma:replacement-Lemmas-Memoirs} that
    $$
    \sup_I\frac{\mu(S(I))}{\om(S(I))^{\frac1p+\frac1{q'}}}
    \lesssim\left\|\frac{\widetilde{\mathcal{T}}_\mu(\cdot)}{\om(S(\cdot))^{\frac{1}{p}+\frac{1}{q'}-1}}\right\|_{L^\infty},
    $$
and hence (ii)$\Rightarrow$(iv).

If now $\mu$ is a $\frac{s(p+q')}{pq'}$-Carleson measure for $A^s_\om$, that is, $\mu$ is a
 $1$-Carleson measure for $A^{\frac{pq'}{p+q'}}_\om$ by \cite[Theorem~1]{PelRatEmb}, then Lemma~\ref{fubini},
 \cite[Theorem~3]{PelRatSie2015} and H\"older's inequality
yield
    \begin{equation*}
    \begin{split}
    \left|\langle \mathcal{T}_\mu(f),g \rangle_{A^2_\om}\right|
    &\le\int_{\D}|f(z)g(z)|\,d\mu(z)
    \lesssim\|Id\|_{A^\frac{pq'}{p+q'}_\om \to L^1_\mu}\left(\int_{\D}|f(z)g(z)|^\frac{pq'}{p+q'}\om(z)\,dA(z)\right)^\frac{p+q'}{pq'}\\
    &\lesssim\|Id\|_{A^\frac{pq'}{p+q'}_\om \to L^1_\mu}\left\|f\right\|_{A^p_\om}\left\|g\right\|_{A^{q'}_\om}
    \end{split}
    \end{equation*}
for all polynomials $f$ and $g$. Since polynomials are dense in both $A^{p}_\om$ and $A^{q'}_\om$, and $(A^q_\om)^\star\simeq A^{q'}_\om$ by Theorem~\ref{th:duality}, it follows that $\mathcal{T}_\mu:A^p_\om\to A^q_\om$ is bounded and $\|\Tm\|_{A^p_\om\to A^q_\om}\lesssim\|Id\|_{A^\frac{pq'}{p+q'}_\om\to L^1_\mu}$. This is the right upper bound for $s=\frac{pq'}{p+q'}$, and the general case follows by an application of \cite[Theorem~3]{PelRatSie2015}.
\end{Prf}

\medskip

Now we turn to compact Toeplitz operators.

\begin{proposition}\label{pr:mainpq}
Let $1<p\le q<\infty$ and $\om\in\R$. If $T:A^p_\om\to A^q_\om$ is a
compact linear operator, then
    $$
    \lim_{|z|\to1^-}\frac{\widetilde{T}(z)}{\om(S(z))^{\frac{1}{p}+\frac{1}{q'}-1}}=0.
    $$
\end{proposition}

\begin{proof}
Since $b^\om_{p,z}\to 0$ weakly in
$A^p_\om$, as $|z|\to 1^-$, by Lemma~\ref{le:3}, and $T:A^p_\om\to
A^q_\om$ is compact, and in particular completely continuous, by the hypothesis, we deduce
    $$
    \left\|T\left(b^\om_{p,z}\right)\right\|_{A^q_\om}\to 0,\quad |z|\to 1^-.
    $$
By H\"older's inequality this implies
    $$
    \left|\left\langle T\left(b^\om_{p,z}\right),b^\om_{q',z}\right\rangle_{A^2_\om}\right|\to 0,\quad|z|\to 1^-.
    $$
Moreover, by Theorem~\ref{Lemma:kernels},
    \begin{equation*}
    \begin{split}
    \|B^\om_z\|_{A^p_\om}\|B^\om_z\|_{A^{q'}_\om}
    &\asymp\frac{1}{\widehat{\om}(z)^{1-\frac1p}(1-|z|)^{1-\frac1p}}\frac{1}{\widehat{\om}(z)^{1-\frac1{q'}}(1-|z|)^{1-\frac1{q'}}}\\
    &\asymp\|B_z^\om\|_{A^2_\om}^2\frac{1}{\widehat{\om}(z)^{1-\frac1p-\frac1{q'}}(1-|z|)^{1-\frac1p-\frac1{q'}}}
    \asymp \|B_z^\om\|_{A^2_\om}^2
    \frac{1}{\om(S(z))^{1-\frac1p-\frac1{q'}}},
    \end{split}
    \end{equation*}
and hence
    \begin{equation*}
    \begin{split} \left|\widetilde{T}(z)\right|\om(S(z))^{1-\frac1p-\frac1{q'}}
    &=\left|\left\langle T\left(b^\om_{z}\right),b^\om_{z}\right\rangle_{A^2_\om}\right|\om(S(z))^{1-\frac1p-\frac1{q'}}\\ &\asymp\left|\left\langle T\left(b^\om_{p,z}\right),b^\om_{q',z}\right\rangle_{A^2_\om}\right|
    \to 0,\quad|z|\to 1^-,
    \end{split}
    \end{equation*}
and the assertion is proved.
\end{proof}

The following result is the analogue of Theorem~\ref{th:tubounpq} for compact Toeplitz operators.

\begin{theorem}\label{th:tucompactpq}
Let $1<p\le q<\infty$, $\om\in\R$ and $\mu$ be a positive Borel measure on $\D$. Then the following statements are equivalent:
    \begin{enumerate}
    \item[(i)] $\mathcal{T}_\mu:A^p_\om\to A^q_\om$ is compact;
    \item[(ii)] $\lim_{|z|\to 1^-}\frac{\widetilde{\mathcal{T}}_\mu(z)}{\om(S(z))^{\frac{1}{p}+\frac{1}{q'}-1}}=0$;
    \item[(iii)] $Id:A^s_\om\to L^{\frac{s(p+q')}{pq'}}_\mu$ is compact for some (equivalently for all)
    $0<s<\infty$;
    \item[(iv)]
    $\lim_{|I|\to 0}\frac{\mu(S(I))}{\om(S(I))^{\frac1p+\frac1{q'}}}=0$.
    \end{enumerate}
\end{theorem}

\begin{proof}
The equivalence (iii)$\Leftrightarrow$(iv) follows from \cite[Theorem~3]{PelRatSie2015}, see also
\cite[Theorem~2.1]{PelRat}. If $\mathcal{T}_\mu:A^p_\om\to A^q_\om$
is compact, then $\lim_{|z|\to
1^-}\frac{\widetilde{\mathcal{T}}_\mu(z)}{\om(S(z))^{\frac{1}{p}+\frac{1}{q'}-1}}=0$
by Proposition~\ref{pr:mainpq}. Assume next that (ii) is satisfied,
and let $\d=\delta(\om)\in(0,1)$ be that of Lemma~\ref{le:kersquare}. By the proof of Theorem~\ref{th:tubounpq},
there exists a constant $C=C(\om)>0$ such that
$\mu(S(z_{\delta}))\le C\widetilde{\mathcal T}_\mu(z)\om(S(z))$
for all $z\in\D\setminus\{0\}$. By applying Lemma~\ref{Lemma:replacement-Lemmas-Memoirs}, and letting $|z|\to1^-$, it follows by the assumption (ii)
that $\lim_{|z|\to1^-}\frac{\mu(S(z))}{\om(S(z))^{\frac{p+q'}{pq'}}}=0$,
and thus $Id:A^s_\om\to L^\frac{s(p+q')}{pq'}_\mu$ is compact by
\cite[Theorem~3]{PelRatSie2015}.

Assume now that $Id:A^s_\om\to L^\frac{s(p+q')}{pq'}_\mu$ is compact
for some (equivalently for all) $0<s<\infty$. Then, by
\cite[Theorem~3]{PelRatSie2015}, $Id:A^p_\om\to
L^\frac{p+q'}{q'}_\mu$ and $Id:A^{q'}_\om\to L^\frac{p+q'}{p}_\mu$
are compact. Let $\{f_n\}$ be a bounded sequence in $A^p_\om$. Then
the proof of \cite[Theorem~2.1]{PelRat} shows that there exists a
subsequence $\{f_{n_k}\}$ and $f\in A^p_\om$ such that
    $
    \lim_{k\to\infty}\|f_{n_k}-f\|_{L^\frac{p+q'}{q'}_\mu}=0.
    $
Write $\mu_r=\chi_{D(0,r)}\mu$ for $0<r<1$. Then Theorem~\ref{th:tubounpq} yields
    \begin{equation*}
    \begin{split}
    \|\mathcal T_\mu(f_{n_k}-f)\|_{A^q_\om}
    &\le\|\mathcal T_{\mu_r}(f_{n_k}-f)\|_{A^q_\om}+\|(\mathcal T_\mu-\mathcal T_{\mu_r})(f_{n_k}-f)\|_{A^q_\om}\\
    &\lesssim\|\mathcal T_{\mu_r}(f_{n_k}-f)\|_{A^q_\om}+\|\mathcal T_\mu-\mathcal T_{\mu_r}\|_{A^p_\om\to A^q_\om},
    \end{split}
    \end{equation*}
where
    \begin{equation*}
    \begin{split}
    \|\mathcal T_\mu-\mathcal T_{\mu_r}\|_{A^p_\om\to A^q_\om}
    &\lesssim\sup_I\frac{\mu(S(I)\setminus D(0,r))}{\om(S(I))^{\frac1p+\frac1{q'}}}
    \lesssim\sup_{|I|\le1-r}\frac{\mu(S(I))}{\om(S(I))^{\frac1p+\frac1{q'}}}\to0,\quad r\to1^-,
    \end{split}
    \end{equation*}
by Theorem~\ref{th:tubounpq} and \cite[Theorem~3]{PelRatSie2015}, because $Id:A^1_\om\to L^\frac{p+q'}{pq'}_\mu$ is compact by the
hypothesis. Moreover, \eqref{tmu}, Theorem~\ref{th:tubounpq} and
H\"older's inequality yield
    \begin{equation*}
    \begin{split}
    \left|\langle \mathcal{T}_{\mu_r}(f_{n_k}-f),g \rangle_{A^2_\om}\right|
    &\le\int_{\D}|(f_{n_k}-f)(z)g(z)|\,d\mu_r(z)
    \le\|f_{n_k}-f\|_{L^\frac{p+q'}{q'}_{\mu_r}}\|g\|_{L^{\frac{p+q'}{p}}_{\mu_r}}\\
    &\le\|f_{n_k}-f\|_{L^\frac{p+q'}{q'}_{\mu_r}}\|Id\|_{A^{q'}_\om\to L^{\frac{p+q'}{p}}_{\mu_r}}\|g\|_{A^{q'}_\om}\\
    &\le\|f_{n_k}-f\|_{L^\frac{p+q'}{q'}_{\mu}}\|Id\|_{A^{q'}_\om\to L^{\frac{p+q'}{p}}_{\mu}}\|g\|_{A^{q'}_\om}.
    \end{split}
    \end{equation*}
Since $(A^q_\om)^\star \simeq A^{q'}_\om$ by Theorem~\ref{th:duality}, we obtain
    \begin{equation*}
    \begin{split}
    \|\mathcal{T}_{\mu_r}(f_{n_k}-f)\|_{A^q_\om}
    &\asymp\sup_{\left\{g:\|g\|_{A^{q'}_\om}\le 1\right\}}\left|\langle \mathcal{T}_{\mu_r}(f_{n_k}-f),g \rangle_{A^2_\om}\right|\\
    &\le\|Id\|_{A^{q'}_\om\to L^{\frac{p+q'}{p}}_\mu}\|f_{n_k}-f\|_{L^\frac{p+q'}{q'}_\mu}\to0,\quad k\to\infty.
    \end{split}
    \end{equation*}
Thus $\mathcal{T}_\mu:A^p_\om\to A^q_\om$ is compact, and the proof
is complete.
\end{proof}

\subsection{Case $1<q<p<\infty$}

We begin with constructing appropriate test functions to be used in the proof of Theorem~\ref{th:qmenorp}. To do this, some notation is needed. The Euclidean discs are denoted by $D(a,r)=\{z\in\C:|a-z|<r\}$. A sequence
$Z=\{z_k\}_{k=0}^\infty\subset\D$ is called separated if it is separated in the pseudohyperbolic metric, it is an $\e$-net for $\e\in(0,1)$ if
$\D=\bigcup_{k=0}^\infty \Delta(z_k,\e)$, and finally it is a
$\delta$-lattice if it is a $5\delta$-net and separated with
constant $\delta/5$.

\begin{proposition}\label{pr:atomicdec}
Let $1<p<\infty$, $\om\in\R$ and
$\{z_j\}_{j=1}^\infty\subset\D\setminus\{0\}$ be a separated sequence. Then $F=\sum_{j=1}^\infty c_jb^\om_{p,z_j}\in A^p_\om$ with
$\|F\|_{A^p_\om}\lesssim\|\{c_j\}_{j=1}^\infty\|_{\ell^p}$ for all $\{c_j\}_{j=1}^\infty\in\ell^p$.
\end{proposition}

\begin{proof}
Let $\{c_j\}_{j=1}^\infty\in\ell^p$, $0<r<1$ and $z\in\overline{D(0,\r)}$ with $0<\r<1$. Then H\"older's inequality and
Theorem~\ref{Lemma:kernels} yield
    $$
    \left|\sum_{j=1}^\infty c_jb^\om_{p,z_j}(z)\right|
    \lesssim\|\{c_j\}_{j=1}^\infty\|_{\ell^p}\left(\sum_{j=1}^\infty\om(\Delta(z_j,r))|B^\om_{z_j}(z)|^{p'}\right)^{1/p'}
    \le C(\om,\r) \|\{c_j\}_{j=1}^\infty\|_{\ell^p}\om(\D),
    $$
and hence $F\in\H(\D)$. Moreover, by H\"older's inequality, Theorem~\ref{Lemma:kernels}, \eqref{eq:discocaja}, the subharmonicity of $|g|^{p'}$
and \eqref{eq:r2},
    \begin{equation*}
    \begin{split}
    \left|\langle g, F\rangle_{A^2_{\om}}\right|
    &=\left|\sum_{j=1}^\infty\overline{c_j}\frac{g(z_j)}{\|B^\om_z\|_{A^p_\om}}\right|
    \lesssim\|\{c_j\}_{j=1}^\infty\|_{\ell^p}\left(\sum_{j=1}^\infty \om(\Delta(z_j,r))|g(z_j)|^{p'}\right)^{1/p'}\\
    &\lesssim\|\{c_j\}_{j=1}^\infty\|_{\ell^p}\left(\sum_{j=1}^\infty\om(z_j) \int_{\Delta(z_j,r)}|g(z)|^{p'}\,dA(z)\right)^{1/p'}\\
    &\asymp\|\{c_j\}_{j=1}^\infty\|_{\ell^p}\left(\sum_{j=1}^\infty \int_{\Delta(z_j,r)}|g(z)|^{p'} \om(z)\,dA(z)\right)^{1/p'}
    \lesssim\|\{c_j\}_{j=1}^\infty\|_{\ell^p}\| g\|_{A^{p'}_\om},
    \end{split}
    \end{equation*}
where in the last step the fact that each $z\in\D$ belongs to at
most $N$ of the discs $\Delta(z_j,r)$ is also used. Therefore $F$
defines a bounded linear functional on $A^{p'}_\om$ with norm
bounded by a constant times $\|\{c_j\}_{j=1}^\infty\|_{\ell^p}$.
Since $(A^{p'}_\om)^\star\simeq A^{p}_\om$ by
Theorem~\ref{th:duality}, this implies $F\in A^p_\om$ with
$\|F\|_{A^p_\om}\lesssim\|\{c_j\}_{j=1}^\infty\|_{\ell^p}$.
\end{proof}

\begin{Prf}{\em{Theorem~\ref{th:qmenorp}.}}
Write $x=x(p,q)= p+1-\frac{p}{q}$ for short. Assume first (ii). Take
 $\{a_j\}_{j=1}^\infty\subset\D\setminus\{0\}$ a separated sequence. Then
Proposition~\ref{pr:atomicdec} gives
    \begin{equation*}
    \left\| \Tm\left(\sum_{j=1}^\infty
    c_jb^\om_{p,a_j}\right)\right\|^q_{A^q_\om}\lesssim
    \|\Tm\|^q_{A^p_\om\to A^q_\om} \|\{c_j\}_{j=1}^\infty\|^q_{\ell^p}.
    \end{equation*}
By replacing $c_k$ by $r_k(t)c_k$, where $r_k$
denotes the $k$th Rademacher function, and applying Khinchine's inequality, we deduce
    \begin{equation}
    \begin{split}\label{12}
    \|\Tm\|^q_{A^p_\om\to A^q_\om} \|\{c_j\}_{j=1}^\infty\|^q_{\ell^p}
    &\gtrsim \int_\D \left(\sum_{j=1}^\infty
    |c_j|^2||\Tm(b^\om_{p,a_j})(z)|^2 \right)^{q/2}\om(z)\,dA(z) \\ &
    \gtrsim \sum_{j=1}^\infty |c_j|^q\int_{\Delta(a_j,s)}
    |\Tm(b^\om_{p,a_j})(z)|^q \om(z)\,dA(z),\quad 0<s<1,
    \end{split}
    \end{equation}
where in the last step the fact that each $z\in\D$ belongs to at
most $N=N(s)$ of the discs $\Delta(a_j,s)$ is also used. By using
the subharmonicity of $|\Tm(b^\om_{p,a_j})|^q$ together with
\eqref{eq:r2} and \eqref{eq:discocaja}, and then applying
Lemma~\ref{b6} and Theorem~\ref{Lemma:kernels}, we obtain
    \begin{equation*}
    \begin{split}
    \int_{\Delta(a_j,s)} |\Tm(b^\om_{p,a_j})(z)|^q \om(z)\,dA(z) &
    \gtrsim \om\left(\Delta(a_j,s)\right)  |\Tm(b^\om_{p,a_j})(a_j)|^q\\
    &=\frac{\om\left(\Delta(a_j,s)\right)}{{\|B^\om_{a_j}\|^q_{A^p_\om}}}\left(\int_{\D}|B^\om_{a_j}(\zeta)|^2\,d\mu(\z)\right)^q\\
    &\ge\frac{\om\left(\Delta(a_j,s)\right)}{{\|B^\om_{a_j}\|^q_{A^p_\om}}}\left(\int_{\Delta(a_j,s)}|B^\om_{a_j}(\zeta)|^2 \,d\mu(\z)\right)^q\\
    &\gtrsim\frac{\om\left(\Delta(a_j,s)\right)}{{\|B^\om_{a_j}\|^q_{A^p_\om}}} \mu\left( \Delta(a_j,s)\right)^q|B^\om_{a_j}(a_j)|^{2q}\\
    &\asymp\left(\frac{\mu\left(\Delta(a_j,s)\right)}{\om\left(\Delta(a_j,s)\right)^{1+\frac{1}{p}-\frac{1}{q}}}\right)^q, \quad
    0<s\le r(\om),
    \end{split}
    \end{equation*}
where $r(\om)$ is that of Lemma~\ref{b6}.
This together with \eqref{12} yields
    \begin{equation}\label{eq:c4}
    \sum_{j=1}^\infty |c_j|^q \left(\frac{\mu\left(\Delta(a_j,s)\right)}{\om\left(\Delta(a_j,s)\right)^{1+\frac{1}{p}-\frac{1}{q}}}\right)^q
    \lesssim\|\Tm\|^q_{A^p_\om\to
    A^q_\om}\|\{c_j\}_{j=1}^\infty\|^q_{\ell^p},\quad  0<s\le r(\om).
    %=\|\Tm\|^q_{A^p_\om\to A^q_\om}\|\{|c_j|^q\}_{j=1}^\infty\|_{\ell^{\frac{p}{q}}}.
    \end{equation}
Let now $s\in (r(\om),1)$ and
    $Z=\{z_j\}_{j=1}^\infty\subset\D\setminus\{0\}$ a $\delta$-lattice
    with $5\delta\le r(\om)$. For each $z_j$ choose
    $N=N(s,r(\om))$ points $z_{k,j}$ of the
    $\delta$-lattice  $Z$
    such that
    $\Delta(z_j,s)\subset \cup_{k=1}^N \Delta(z_{k,j}, r(\om))$. Then, by
    \eqref{eq:r2},
    \eqref{eq:discocaja} and \eqref{eq:c4},
    \begin{equation*}
    \begin{split}
    \sum_{j=1}^\infty |c_j|^q \left(\frac{\mu\left(\Delta(z_j,s)\right)}{\om\left(\Delta(z_j,s)\right)^{1+\frac{1}{p}-\frac{1}{q}}}\right)^q
    &\lesssim \sum_{j=1}^\infty \sum_{k=1}^N|c_j|^q \left(
    \frac{\mu\left(\Delta(z_{k,j},r(\om))\right)}{\om\left(\Delta(z_{k,j},
    r(\om))\right)^{1+\frac{1}{p}-\frac{1}{q}}} \right)^q\\
    &=\sum_{k=1}^N\sum_{j=1}^\infty
    |c_j|^q \left( \frac{\mu\left(\Delta(z_{k,j},
    r(\om))\right)}{\om\left(\Delta(z_{k,j},
    r(\om))\right)^{1+\frac{1}{p}-\frac{1}{q}}} \right)^q\\
    &\lesssim \|\Tm\|^q_{A^p_\om\to
    A^q_\om}\|\{c_j\}_{j=1}^\infty\|^q_{\ell^p}.
    \end{split}
    \end{equation*}
Therefore \eqref{eq:c4} holds for each $0<s<1$ and any $\delta$-lattice
$\{z_j\}_{j=1}^\infty\subset\D\setminus\{0\}$ with $5\delta\le
r(\om)$. The classical duality relation
$\left(\ell^{p/q}\right)^\star\simeq \ell^{\frac{p}{p-q}}$ now
implies
    $$
    \sum_{j=1}^\infty \left(\frac{\mu\left(\Delta(z_j,s)\right)}{\om\left(\Delta(z_j,s)\right)}\right)^{\frac{qp}{p-q}}
    \om\left(\Delta(z_j,s)\right)
    =\sum_{j=1}^\infty\left(\frac{\mu\left(\Delta(z_j,s)\right)}{\om\left(\Delta(z_j,s)\right)^{1+\frac{1}{p}-\frac{1}{q}}}\right)^{\frac{qp}{p-q}}
    \lesssim\|\Tm\|^{\frac{pq}{p-q}}_{A^p_\om\to A^q_\om}.
    $$
Let $0<r<1$, and choose $s=s(r,\d)\in(0,1)$ such that
$\Delta(z,r)\subset\Delta(z_j,s)$ for all $z\in\Delta(z_j,5\d)$ and
$j\in\N$. Then \eqref{eq:r2} and \eqref{eq:discocaja} imply
    \begin{equation}\label{eq:c3}
    \begin{split}
    \|\widehat{\mu}_r\|_{L^\frac{qp}{p-q}_\om}^\frac{qp}{p-q}
    &\le\sum_{j=1}^\infty\int_{\Delta(z_j,5\d)}\widehat{\mu}_r(z)^\frac{qp}{p-q}\om(z)\,dA(z)\\
    &\asymp\sum_{j=1}^\infty\frac{\om(z_j)}{(\om(z_j)(1-|z_j|)^2)^\frac{qp}{p-q}}\int_{\Delta(z_j,5\d)}\mu(\Delta(z,r))^{\frac{qp}{p-q}}\,dA(z)\\
    &\asymp\sum_{j=1}^\infty\frac{\mu(\Delta(z_j,s))^{\frac{qp}{p-q}}}{(\om(z_j)(1-|z_j|)^2)^{\frac{qp}{p-q}-1}}
    \asymp\sum_{j=1}^\infty\left(\frac{\mu\left(\Delta(z_j,s)\right)}{\om\left(\Delta(z_j,s)\right)^{1+\frac{1}{p}-\frac{1}{q}}}\right)^{\frac{qp}{p-q}}.
    \end{split}
    \end{equation}
Thus (iii) is satisfied and
$\|\widehat{\mu}_r\|_{L^{\frac{qp}{p-q}}_{\om}}\lesssim
\|\Tm\|_{A^p_\om\to A^q_\om}$ for each fixed $0<r<1$.

Assume next (iii). By using the subharmonicity of $|f|^x$ together with \eqref{eq:r2} and \eqref{eq:discocaja}, and then Fubini's theorem and H\"older's inequality we deduce
    \begin{equation*}
    \begin{split}
    \int_{\D}|f(z)|^x\,d\mu(z)
    &\lesssim\int_{\D}\frac{\int_{\Delta(z,r)}|f(\z)|^x\om(\z)\,dA(\z)}{\om\left(\Delta(z,r)\right)}\,d\mu(z)\\
    &\asymp\int_{\D}\frac{\mu\left( \Delta(\z,r)\right)}{\om\left(\Delta(\z,r)\right)} |f(\z)|^x\om(\z)\,dA(\z)
    \le\|f\|^x_{A^p_\om}\|\widehat{\mu}_r\|_{L^\frac{qp}{p-q}_\om}.
    %\left(\int_{\D}\left(\frac{\mu\left(\Delta(\z,r)\right)}{\om\left(\Delta(\z,r)\right)}\right)^{\frac{qp}{p-q}}
    %\om(\z)\,dA(\z)\right)^{\frac{p-q}{qp}}.
    \end{split}\end{equation*}
Therefore $\mu$ is a $\left(p+1-\frac{p}{q}\right)$-Carleson
measure for $A^p_\om$, that is, (iv) is satisfied, and $\|Id\|^x_{A^p_\om\to L^{q}_\mu}\lesssim\|\widehat{\mu}_r\|_{L^{\frac{qp}{p-q}}_\om}$. In fact, it follows from \cite[Theorem~3.2]{OC} and \cite[Lemma~1.4]{PelRat} that
$Id:A^p_\om\to L^x_\mu$ is bounded if and only if (iii) is satisfied.

The equivalence (iv)$\Leftrightarrow$(v) follows from
\cite[Theorem~3]{PelRatSie2015}.

Let us now prove (iv)$\Rightarrow$(ii). Since $\mu$ is an
$x$-Carleson measure for $A^p_\om$ by the hypothesis~(iv),
Lemma~\ref{fubini}, H\"older's inequality and \cite[Theorem~3]{PelRatSie2015} together with the equality
$\frac{x}{p}=\frac{x'}{q'}$ give
    \begin{equation}
    \begin{split}\label{eq:bounded1}
    \left|\langle \mathcal{T}_\mu(f),g \rangle_{A^2_\om}\right|
    &\le\int_{\D}|f(z)g(z)|\,d\mu(z)
    \le \|f\|_{L^x_\mu}\|g\|_{L^{x'}_\mu}\\
    &\le\|Id\|_{A^p_\om\to L^x_\mu}\|Id\|_{A^{q'}_\om\to L^{x'}_\mu}\left\|f\right\|_{A^p_\om}\left\|g\right\|_{A^{q'}_\om}\\
    &\lesssim\|Id\|^{x}_{A^p_\om \to L^x_\mu}
    \left\|f\right\|_{A^p_\om}\left\|g\right\|_{A^{q'}_\om}
    \end{split}
    \end{equation}
for polynomials $f$ and $g$. Since polynomials are dense in both $A^{q'}_\om$ and $ A^{p}_\om$, and $(A^q_\om)^\star\simeq A^{q'}_\om$ by Theorem~\ref{th:duality}, it follows that $\mathcal{T}_\mu:A^p_\om\to A^q_\om$
is bounded and $\|\Tm\|_{A^p_\om\to
A^q_\om}\lesssim\|Id\|^{x}_{A^p_\om \to L^x_\mu}$.

The implication (ii)$\Rightarrow$(i) follows by a general argument.
Namely, for $1<p<\infty$, $A^p_\om$ is isomorphic to $\ell^p$ by
\cite[Corollary~2.6]{LuskyStud87} and Lemma~\ref{Lemma:kernels}.
Moreover, each bounded linear operator $L:\ell^p\to\ell^q$,
$0<q<p<\infty$, is compact by \cite[Theorem~I. 2.7,
p.~31]{LinTzaClassical}. Thus $\Tm:A^p_\om\to A^q_\om$ is compact.

It remains to prove (iii)$\Leftrightarrow$(vi) and the equivalence of norms
$\|\widehat{\mu}_r\|_{L^{\frac{qp}{p-q}}_\om}\asymp
\|\widetilde{\mathcal{T}}_\mu\|_{L^{\frac{qp}{p-q}}_\om}$ for each fixed $r\in(0,1)$. Assume $\widetilde{\mathcal{T}}_\mu\in L^\frac{pq}{p-q}_\omega$, and let first $r\in(0,r(\omega)]$, where $r(\omega)$ is that of
Lemma~\ref{b6}. Then Lemma~\ref{fubini} and Theorem~\ref{Lemma:kernels} give
    \begin{equation}\label{eq:viddu}
    \begin{split}
    \widetilde{\mathcal{T}}_\mu(z)
    &=\int_\D |b_z^\omega(\z)|^2\,d\mu(\z)
    \ge\int_{\Delta(z,r)}|b_z^\omega(\z)|^2\,d\mu(\z)
    \asymp|b_z^\omega(z)|^2 \mu(\Delta(z,r))
    \asymp\widehat{\mu}_r(z).%\frac{\mu(\Delta(z,r))}{\omega(\Delta(z,r))},
    \end{split}
    \end{equation}
Hence $\widehat{\mu}_r\in L^\frac{pq}{p-q}_\omega$ and
$\|\widehat{\mu}_r\|_{L^{\frac{qp}{p-q}}_\om}\lesssim
\|\widetilde{\mathcal{T}}_\mu\|_{L^{\frac{qp}{p-q}}_\om}$. Let now
$r\in(r(\om),1)$, and let $\{z_j\}$ be a $\d$-lattice. Further, let
$s=s(r,\d)$ be that of \eqref{eq:c3}, and choose $r'=r'(r(\om))$
such that $\Delta(z,r')\subset\Delta(w,r(\om))$ for all
$z\in\Delta(w,r')$ and $w\in\D$. Furthermore, choose
$z_j^n\in\Delta(z_j,s)$, $n=1,\ldots,N$, such that
$\Delta(z_j,s)\subset\cup_{n=1}^N\Delta(z_j^n,r')$ for all $j$ and
$\inf_j\min_{n\ne m}\varrho(z_j^n,z_j^m)>0$. Then \eqref{eq:c3},
Lemma~\ref{Lemma:replacement-Lemmas-Memoirs}, Lemma~\ref{b6},
Theorem~\ref{Lemma:kernels} and \eqref{Eq:BerezinTransformation}
 yield
    \begin{equation*}
    \label{Eq:pilipali}
    \begin{split}
    \|\widehat{\mu}_r\|_{L^\frac{qp}{p-q}_\om}^\frac{qp}{p-q}
    &\lesssim\sum_{j=1}^\infty\frac{\mu(\Delta(z_j,s))^{\frac{qp}{p-q}}}{\om(\Delta(z_j,s))^{\frac{qp}{p-q}-1}}
    \lesssim\sum_{j=1}^\infty\sum_{n=1}^N\om(\Delta(z_j,s))\left(\frac{\mu(\Delta(z_j^n,r'))}{\om(\Delta(z_j,s))}\right)^{\frac{qp}{p-q}}\\
    &\asymp\sum_{j=1}^\infty\sum_{n=1}^N\om(\Delta(z^n_j,r'))\widehat{\mu}_{r'}(z_j^n)^{\frac{qp}{p-q}}\\
    &\asymp\sum_{j=1}^\infty\sum_{n=1}^N\int_{\Delta(z^n_j,r')}\left(\int_{\Delta(z^n_j,r')}|b^\om_z(\z)|^2\,d\mu(\z)\right)^{\frac{qp}{p-q}}\om(z)\,dA(z)\\
    &\le\sum_{j=1}^\infty\sum_{n=1}^N\int_{\Delta(z^n_j,r')}\left(\widetilde{\mathcal{T}}_\mu(z)\right)^{\frac{qp}{p-q}}\om(z)\,dA(z)
    \asymp\|\widetilde{\mathcal{T}}_\mu\|_{L^\frac{qp}{p-q}_\om}^\frac{qp}{p-q}.
    \end{split}
    \end{equation*}

Now assume (iii), and let $h$ be a positive subharmonic function in $\D$. Then \eqref{eq:r2}, \eqref{eq:discocaja} and Fubini's theorem yield
    \begin{equation*}
    \begin{split}
    \int_\D h(z)\,d\mu(z)
    &\lesssim\int_{\D}\left(\frac{1}{(1-|z|)^2}\int_{\Delta(z,r)} h(\z)\,dA(\z)\right)d\mu(z)\\
    &\asymp\int_{\D}\left(\int_{\Delta(z,r)}h(\z)\frac{\om(\z)}{\omega(\Delta(\z,r))}\,dA(\z)\right)d\mu(z)
    =\int_\D h(\z)\widehat{\mu}_r(\z)\omega(\z)\,dA(\z).
    \end{split}
    \end{equation*}
This together with \eqref{Eq:BerezinTransformation},
Theorem~\ref{Lemma:kernels}, and Lemma~\ref{Blochnorm} yield
\begin{equation*}
    \begin{split}
    \widetilde{\mathcal{T}}_\mu(z)
    &=\int_\D|b_z^\omega(\z)|^2\,d\mu(\z)
    \lesssim\int_\D|b_z^\omega(\z)|^2\widehat{\mu}_r(\z)\omega(\z)\,dA(\z)\\
    &=\int_\D|B_z^\omega(\z)|\frac{|B_z^\omega(\z)|}{\|B_z^\omega\|_{A^2_\omega}^2}\widehat{\mu}_r(\z)\omega(\z)\,dA(\z)\\
    &\le\frac{\|B_z^\omega\|_{H^\infty}}{\|B_z^\omega\|_{A^2_\omega}^2}\int_\D|B_z^\omega(\z)|\widehat{\mu}_r(\z)\omega(\z)\,dA(\z)
    \asymp P^+_\omega(\widehat{\mu}_r)(z).
    \end{split}
    \end{equation*}
But $P^+_\omega:L^\frac{pq}{p-q}_\omega \to L^\frac{pq}{p-q}_\omega$ is bounded by \cite[Theorem~5]{Twoweight} because $\frac{pq}{p-q}>1$ and $\om\in\R$, and hence
    $$
    \|\widetilde{\mathcal{T}}_\mu\|_{L^\frac{pq}{p-q}_\omega}
    \lesssim\|P^+_\omega(\widehat{\mu}_r)\|_{L^\frac{pq}{p-q}_\omega}
    \lesssim\|\widehat{\mu}_r\|_{L^\frac{pq}{p-q}_\omega}<\infty.
    $$
This finishes the proof.
\end{Prf}

\section{Schatten class Toeplitz operators}\label{sec4}

The purpose of this section is to prove Theorem~\ref{MTMSchatten}, or more precisely,
 the last part of it, and then show that it can not be extended to the whole class $\DD$ of doubling weights.
We begin with some necessary notation and definitions, and
preliminary results which are well-known in the setting of standard
weights~\cite{Zhu}.

Let $H$ be a separable Hilbert space. For any non-negative integer
$n$, the $n$:th singular value of a bounded operator $T:H\to
H$ is defined by
    $$
    \lambda_n(T)=\inf\left\{\|T-R\|:\,\text{rank}(R)\le
    n\right\},
    $$
where $\|\cdot\|$ denotes the operator norm. It is clear that
    $$
    \|T\|=\lambda_0(T)\ge\lambda_1(T)\ge\lambda_2(T)\ge\dots\ge 0.
    $$
For $0<p<\infty$, the Schatten $p$-class $\mathcal{S}_ p(H)$
\index{$\mathcal{S}_ p(H)$} consists of those compact operators
$T:H\to H$ whose sequence of singular values
$\{\lambda_n\}_{n=0}^\infty$ belongs to the space~$\ell^p$ of
$p$-summable sequences. For $1\le p<\infty$, the Schatten $p$-class
$\mathcal{S}_p(H)$ is a Banach space with respect to the norm
$|T|_p=\|\{\lambda_n\}_{n=0}^\infty\|_{\ell^p}$. Therefore all
finite rank operators belong to every $\mathcal{S}_p(H)$, and the
membership of an operator in $\mathcal{S}_p(H)$ measures in some
sense the size of the operator. We refer to~\cite{DS} and
\cite[Chapter~1]{Zhu} for more information about $\mathcal{S}_p(H)$.

The first auxiliary result is well known and its proof is straightforward, so the details are omitted.

\begin{letterlemma}\label{b0}
Let $H$ be a separable Hilbert space and $T:H\to H$ a bounded linear
operator such that $\sum_{n}|\langle T(e_n),e_n\rangle_H|<\infty$
for every orthonormal basis $\{e_n\}$. Then $T:H\to H$ is compact.
\end{letterlemma}

The next result characterizes positive operators in the trace class $\mathcal{S}_1(A^2_\om)$ in terms of their Berezin transforms.

\begin{theorem}\label{b1}
Let $\om\in\DD$ and $T:A^2_\om\to A^2_\om$ a positive operator. Then
$T\in\mathcal{S}_1(A^2_\om)$ if and only if $\wtT\in
L^1_{\om/\om^\star}$. Moreover, the trace of $T$ satisfies
    $$
    tr(T)=\int_{\D}\wtT(z)\|B^\om_z\|^2_{A^2_\om} \om(z)\,dA(z)\asymp\int_\D\widetilde{T}(z)\frac{\om(z)}{\om^\star(z)}\,dA(z).
    $$
\end{theorem}

\begin{proof} The proof is similar to that of
\cite[Theorem~6.4]{Zhu}, and is included for the convenience of the reader.
Fix an orthonormal basis $\{e_n\}_{n=1}^\infty$ for $A^2_\om$. Since
$T$ is positive, \cite[Theorem~1.23]{Zhu} and Lemma~\ref{b0} show
that $T\in\SSS_1(A^2_\om)$ if and only if $\sum_{n=1}^\infty \langle
T(e_n), e_n \rangle_{A^2_\om}<\infty$, and further,
$tr(T)=\sum_{n=1}^\infty \langle T(e_n), e_n \rangle_{A^2_\om}$. Let
$S=\sqrt{T}$. By the reproducing formula \eqref{repfor} and
Parseval's identity, Theorem~\ref{Lemma:kernels} and
Lemma~\ref{Lemma:replacement-Lemmas-Memoirs}, we have
    \begin{equation*}
    \begin{split}
    tr(T)&=\sum_{n=1}^\infty \langle T(e_n), e_n \rangle_{A^2_\om}
    =\sum_{n=1}^\infty\|S(e_n)\|^2_{A^2_\om}\\
    &=\sum_{n=1}^\infty \int_{\D}|S(e_n)(z)|^2\om(z)\,dA(z)
    =\int_{\D}\left(\sum_{n=1}^\infty |S(e_n)(z)|^2\right)\om(z)\,dA(z)\\
    &=\int_{\D}\left(\sum_{n=1}^\infty \left|\langle S(e_n), B^\om_z\rangle_{A^2_\om}\right|^2\right)\om(z)\,dA(z)\\
    &=\int_{\D}\left(\sum_{n=1}^\infty\left|\langle e_n, S(B^\om_z)\rangle_{A^2_\om}\right|^2\right)\om(z)\,dA(z)\\
    &=\int_{\D}\|S(B^\om_z)\|^2_{A^2_\om} \,\om(z)\,dA(z)
    =\int_{\D}\langle T(B^\om_z), B^\om_z\rangle_{A^2_\om} \om(z)\,dA(z)\\
    &=\int_{\D}\wtT(z)\|B^\om_z\|^2_{A^2_\om} \om(z)\,dA(z)
    \asymp \int_\D\wtT(z)\frac{\om(z)}{\om^\star(z)}\,dA(z),
    \end{split}
    \end{equation*}
and the assertion is proved.
\end{proof}

By combining Theorem~\ref{b1} with \cite[Proposition~1.31]{Zhu} we obtain the following result.

\begin{lemma}\label{b3}
Let $\om\in\DD$ and $T:A^2_\om\to A^2_\om$ a positive operator.
\begin{itemize}
\item[\rm(i)] If $1\le p<\infty$ and $T\in\SSS_p(A^2_\om)$, then
$\wtT\in L^p_{\om/\om^\star}$ with $\|\widetilde T\|_{L^p_{\om/\om^\star}}\lesssim|T|_p^p$.
\item[\rm(ii)] If $0<p\le 1$ and
$\wtT\in L^p_{\om/\om^\star}$, then $T\in\SSS_p(A^2_\om)$ with $|T|_p^p\lesssim\|\widetilde T\|_{L^p_{\om/\om^\star}}$.
\end{itemize}
\end{lemma}

Recall that
    $$
    \mathcal{T}_\F(f)(z)=P_\om(f\F)(z)=\int_\D f(\z)\overline{B_z^\om(\z)}\,\Phi(\z) \om(\z) dA(\z),\quad f\in
    A^2_\om,
    $$
for each non-negative function $\Phi$ on $\D$. We next establish a sufficient condition for $\mathcal{T}_\F$ to belong to $\SSS_p(A^2_\om)$ for $1\le p<\infty$.

\begin{proposition}\label{b4}
Let $1\le p<\infty$, $\om\in\DD$ and $\F\in L^p_{\om/\om^\star}$ positive. Then $\mathcal{T}_\F\in\SSS_p(A^2_\om)$ with $|\mathcal{T}_\Phi|_p\lesssim\|\Phi\|_{L^p_{\om/\om^\star}}$.
\end{proposition}

\begin{proof}
We will follow the proof of \cite[Proposition~7.11]{Zhu}. Assume first that $\F$ has compact support in $\D$. Then $\mathcal{T}_\F$ is a positive compact operator with canonical decomposition
    $$
    \mathcal{T}_\F(f)=\sum_{n=1}^\infty \lambda_n\langle f,e_n\rangle_{A^2_\om} e_n,
    $$
where $\{\la_n\}$ is the sequence of eigenvalues of $\mathcal{T}_\F$, and
$\{e_n\}$ is an orthonormal set of $A^2_\om$. Therefore
    $$
    \la_n=\langle \mathcal{T}_\F(e_n),e_n\rangle_{A^2_\om}=\int_{\D} |e_n(z)|^2\F(z)\om(z)\,dA(z),\quad n\in\N,
    $$
by \eqref{tmu}. Since $p\ge1$, the H\"older's inequality yields
    $$
    \la_n^p\le \int_{\D} |e_n(z)|^2\F(z)^p\om(z)\,dA(z),
    $$
and hence
    \begin{equation}\label{sp1}
    \begin{split}
    \sum_{n=1}^\infty \lambda_n^p &\le \int_{\D} \sum_{n=1}^\infty|e_n(z)|^2\F(z)^p\om(z),dA(z)\\
    &\le \int_{\D} B_z^\om(z) \F(z)^p\om(z)\,dA(z)
    \asymp \int_{\D}\F(z)^p\frac{\om(z)}{\omega^\star(z)}\,dA(z)
    \end{split}
    \end{equation}
by Theorem~\ref{Lemma:kernels}. Thus $\mathcal{T}_\Phi\in\mathcal{S}_p(A^2_\om)$.

To prove the general case, assume $\F\in L^p_{\om/\om^\star}$. Then H\"older's inequality and Lemma~\ref{Lemma:replacement-Lemmas-Memoirs} yield
    \begin{equation*}
    \begin{split}
    \lim_{|a|\to 1^-}\frac{\int_{S(a)}\F(z)\om(z)\,dA(z)}{\om(S(a))}
    &\le\lim_{|a|\to1^-}\left(\frac{\int_{S(a)}\F(z)^p\om(z)\,dA(z)}{\om(S(a))}\right)^{\frac{1}{p}}\\
    &\lesssim \lim_{|a|\to 1^-}\left(\int_{S(a)}\F(z)^p\frac{\om(z)}{\om^\star(z)}\,dA(z)\right)^{\frac{1}{p}}=0,
    \end{split}
    \end{equation*}
and hence $\mathcal{T}_\Phi:A^2_\om\to A^2_\om$ is compact by Theorem~\ref{th:tucompactpq}.

Now write $\F_r=\F\chi_{D(0,r)}$, where $\chi_{D(0,r)}$ is the
characteristic function of $D(0,r)$. Arguing as in \eqref{sp1} it
follows that $\{\mathcal{T}_{\F_{r}}\}_{r\in (0,1) }$ is Cauchy in the Banach
space $\left(\SSS_p(A^2_\om),|\cdot|_p\right)$. Hence there exists
$T\in\SSS_p(A^2_\om)$ such that $\lim_{r\to 1^-}|\mathcal{T}_{\F_{r}}-T|_p=0$.
On the other hand, if $f$ is a polynomial and $z\in\D$, then
Lemma~\ref{fubini} and H\"older's inequality yield
    \begin{equation*}
    \begin{split}
    |(\mathcal{T}_{\F_{r}}-\mathcal{T}_{\F})(f)(z)|
    &=|\langle(\mathcal{T}_{\F_{r}}-\mathcal{T}_{\F})(f),B^\om_z\rangle_{A^2_\om}|\\
    &=\left|\int_{r<|\z|<1}f(\z)B^\om_\z(z)\F(\z)\om(\z)\,dA(\z)\right|\\
    &\le C\|f\|_{H^\infty}\int_{r<|\z|<1}\F(\z)\om(\z)\,dA(\z)\\
    &\le C\|f\|_{H^\infty}\left(\int_{r<|\z|<1}\F(\z)^p\frac{\om(\z)}{\om^\star(\z)}\,dA(\z)\right)^\frac1p\\
    &\quad\cdot\left(\int_{\D}\om^\star(\z)^{p'-1}\om(\z)\,dA(\z)\right)^\frac1{p'}\to 0,\quad r\to 1^-,
    \end{split}
    \end{equation*}
where $C=C(z)$ is a constant. Thus $\mathcal{T}_{\F_{r}}(f)\to
\mathcal{T}_{\F}(f)$ pointwise for any polynomial $f$. Since
$\mathcal{T}_{\F_{r}}$ and $\mathcal{T}_{\F}$ are bounded on
$A^2_\om$, and polynomials are dense in $A^2_\om$, we deduce that
$\mathcal{T}_{\F_{r}}(f)\to \mathcal{T}_\Phi(f)$ pointwise for all
$f\in A^2_\om$. Therefore $\mathcal{T}_\F=T\in\SSS_p(A^2_\om)$.
\end{proof}

We will need one more auxiliary result in the proof of Theorem~\ref{MTMSchatten}.

\begin{proposition}\label{b5}
Let $\om\in\R$, $0<r<1$ and $\mu$ be a finite positive Borel measure
on $\D$ such that $\mathcal{T}_{\widehat{\mu}_r}:A^2_\om\to A^2_\om$ is
bounded. Then $\mathcal{T}_\mu:A^2_\om\to A^2_\om$ is bounded, and
there exists $C=C(\om,r)>0$ such that $\langle
\mathcal{T}_{\mu}(f),f\rangle_{A^2_\om}\le C\langle
\mathcal{T}_{\widehat{\mu}_r}(f),f\rangle_{A^2_\om}$ for all $f\in
A^2_\om$.
\end{proposition}

\begin{proof}
Note first that $\mathcal{T}_\mu:A^2_\om\to A^2_\om$
is bounded by Theorem~\ref{th:tubounpq} and
\cite[Theorem~1]{PelRatEmb}, see also \cite[Theorem~3.1 and
Theorem~4.1]{OC}. Let $f$ be a polynomial. Then
    \begin{equation*}
    \begin{split}
    |f(\zeta)|^2
    &\lesssim\frac{1}{(1-|\zeta|)^2}\int_{\Delta(\zeta,r)}|f(z)|^2\,dA(z)\asymp\int_{\Delta(\zeta,r)}\frac{|f(z)|^2}{(1-|z|)^2}\,dA(z),\quad\zeta
    \in\D,
    \end{split}
    \end{equation*}
and hence Fubini's theorem, Lemma~\ref{fubini}, Lemma~\ref{Lemma:replacement-Lemmas-Memoirs} and
\eqref{eq:discocaja} yield
    \begin{equation*}
    \begin{split}
    \langle \mathcal{T}_{\mu}(f),f\rangle_{A^2_\om}
    &=\int_\D |f(\zeta)|^2\,d\mu(\zeta)
    \lesssim \int_\D\left(\int_{\Delta(\zeta,r)}\frac{|f(z)|^2}{(1-|z|)^2}\,dA(z)\right)\,d\mu(\zeta)\\
    &=\int_\D\frac{|f(z)|^2}{(1-|z|)^2\om(z)}\left(\int_{\Delta(z,r)}d\mu(\zeta) \right)\om(z)\,dA(z)\\
    &\asymp \int_\D \frac{|f(z)|^2}{\om(\Delta(z,r))}\left(\int_{\Delta(z,r)}d\mu(\zeta) \right)\om(z)\,dA(z)
    =\langle
    \mathcal{T}_{\widehat{\mu}_r}(f),f\rangle_{A^2_\om}.
    \end{split}
    \end{equation*}
Since $\mathcal{T}_{\widehat{\mu}_r}:A^2_\om\to A^2_\om$ and
$\mathcal{T}_\mu:A^2_\om\to A^2_\om$ are bounded, and polynomials are dense in $A^2_\om$, it follows that
    $$
    \langle\mathcal{T}_{\mu}(f),f\rangle_{A^2_\om}
    \lesssim\langle\mathcal{T}_{\widehat{\mu}_r}(f),f\rangle_{A^2_\om},\quad f\in A^2_\om,
    $$
and the proof is complete.
\end{proof}

\medskip

\begin{Prf}{\em{Theorem~\ref{MTMSchatten}}.}
The conditions (i)--(iii) are equivalent by \cite[Theorem~1]{PR2016/2}, so it suffices to prove the last claim which concerns the Berezin transform.

The assertion is valid for $p=1$ and $\om\in\DD$ by Theorem~\ref{b1}. For $1<p<\infty$,
Lemma~\ref{b3} shows that $\mathcal{T}_\mu\in\SSS_p(A^2_\om)$
implies $\widetilde{\mathcal{T}}_\mu\in L^p_{\om/\om^\star}$ with $\|\widetilde{\mathcal T}_\mu\|_{L^p_{\om/\om^\star}}\lesssim|\mathcal{T}_\mu|_p$. To see the
converse implication, let $r\in(0,r(\om))$, where $r(\om)$ is that of Lemma~\ref{b6}. If $\widetilde{\mathcal{T}}_\mu\in L^p_{\om/\om^\star}$, then $\widehat{\mu}_r\in L^p_{\om/\om^\star}$ with $\|\widehat{\mu}_r\|_{L^p_{\om/\om^\star}}\lesssim\|\widetilde{\mathcal{T}}_\mu\|_{L^p_{\om/\om^\star}}$ by \eqref{eq:viddu}. Therefore
$\mathcal{T}_{\widehat{\mu}_r}\in\SSS_p(A^2_\om)$ by
Proposition~\ref{b4}, which in turn implies
$\mathcal{T}_{\mu}\in\SSS_p(A^2_\om)$ with $|\mathcal{T}_\mu|_p\lesssim\|\widetilde{\mathcal{T}}_\mu\|_{L^p_{\om/\om^\star}}$ by Proposition~\ref{b5} and
\cite[Theorem~1.27]{Zhu}.

Let now $0<p<1$. If $\widetilde{\mathcal{T}}_\mu\in L^p_{\om/\om^\star}$,
then $\mathcal{T}_\mu\in\SSS_p(A^2_\om)$ with $|\mathcal{T}_\mu|_p\lesssim\|\widetilde{\mathcal{T}}_\mu\|_{L^p_{\om/\om^\star}}$ by Lemma~\ref{b3}. Conversely, assume that $\mathcal{T}_\mu\in\SSS_p(A^2_\om)$. Then \eqref{Eq:BerezinTransformation} yields
    \begin{equation*}
    \begin{split}
    (\widetilde{\mathcal{T}}_\mu(z))^p
    &\asymp(\om^\star(z))^p\|B^\om_z\|^{2p}_{L^2_\mu}
    =(\om^\star(z))^p\left(\sum_{R_j\in\Upsilon}\int_{R_j}|B_z^\om(\z)|^2\,d\mu(\zeta)\right)^p\\
    &\le(\om^\star(z))^p\sum_{R_j\in\Upsilon}\left(\frac{\mu(R_j)}{\om^\star(z_j)}\right)^p
    |B_z^\om(\widetilde{z}_{j,z})|^{2p}(\om^\star(z_j))^p,\quad|z|\ge\frac12,
    \end{split}
    \end{equation*}
where $\widetilde{z}_{j,z}\in \overline{R}_j$ such that $\sup_{\z\in
R_j}|B_z^\om(\z)|=|B_z^\om(\widetilde{z}_{j,z})|$. Consequently,
    \begin{equation}\label{eq:chumino}
    \begin{split}
    \|\widetilde{\mathcal{T}}_\mu\|_{L^p_{\om/\om^\star}}^p
    &\le\sum_{R_j\in\Upsilon}\left(\frac{\mu(R_j)}{\om^\star(z_j)}\right)^p(\om^\star(z_j))^p
    \int_\D|B_{z}^\om(\widetilde{z}_{j,z})|^{2p}(\om^\star(z))^{p-1}\om(z)\,dA(z)\\
    &\asymp\sum_{R_j\in\Upsilon}\left(\frac{\mu(R_j)}{\om^\star(z_j)}\right)^p(\om^\star(z_j))^p
    \int_\D|B_{z}^\om(\widetilde{z}_{j,z})|^{2p}\frac{\om^\star(z))^{p}}{(1-|z|)^2}\,dA(z)
    \end{split}
    \end{equation}
because $\om\in\R$.
Now, fix $0<r<1$ and $\delta=\delta(r)\in(0,1)$ such that $\Delta(z,r)\subset \Delta(z_j,\delta)$ for all $z\in R_j$.
Then, by the subharmonicity of $|B_z^\om|^{2p}$ and Fubini's theorem,
    \begin{equation*}
    \begin{split}
    &\int_\D|B_{z}^\om(\widetilde{z}_{j,z})|^{2p}\frac{\om^\star(z)^{p}}{(1-|z|)^2}\,dA(z)\\
    &\lesssim \int_\D\frac{1}{(1-|\widetilde{z}_{j,z}|)^2}
    \left(\int_{\Delta(\widetilde{z}_{j,z},r)}|B_{z}^\om(\z)|^{2p}\,dA(\z)\right)\frac{\om^\star(z)^{p}}{(1-|z|)^2}\,dA(z)\\
    &\lesssim \frac{1}{(1-|z_{j}|)^2}\int_\D\left(\int_{\Delta(z_j,\delta)}|B_{\z}^\om(z)|^{2p}\,dA(\z)\right)\frac{\om^\star(z)^{p}}{(1-|z|)^2}\,dA(z)\\
    &=\frac{1}{(1-|z_{j}|)^2}\int_{\Delta(z_j,\delta)}\left(\int_\D
    |B_{\z}^\om(z)|^{2p}\frac{\om^\star(z)^{p}}{(1-|z|)^2}\,dA(z)\right)\,dA(\z).
    \end{split}
    \end{equation*}
An application of Theorem~\ref{Lemma:kernels} together with Lemma~\ref{Lemma:replacement-Lemmas-Memoirs} and the hypothesis that $\om^\star(\cdot)/(1-|\cdot|)^2$ is a regular weight show that the inner integral above is dominated by a constant times
    $$
    \int_0^{|\z|}\frac{\int_t^1\frac{\om^\star(s)^{p}}{(1-s)^2}\,ds}{\widehat{\om}(t)^{2p}(1-t)^{2p}}\,dt
    %\asymp\int_0^{|\z|}\frac{\int_t^1\frac{\widehat{\om}(s)^{p}(1-s)^{p}}{(1-s)^2}\,ds}{\widehat{\om}(t)^{2p}(1-t)^{2p}}\,dt
    \asymp\int_0^{|\z|}\frac{1}{\widehat{\om}(t)^{p}(1-t)^{p+1}}\,dt
    \asymp\frac{1}{\widehat{\om}(\z)^{p}(1-|\z|)^{p}},
    $$
and hence
    \begin{equation*}
    \begin{split}
    \int_\D|B_{z}^\om(\widetilde{z}_{j,z})|^{2p}\frac{\om^\star(z))^{p}}{(1-|z|)^2}\,dA(z)
    \lesssim\frac{1}{\widehat{\om}(z_j)^{p}(1-|z_j|)^{p}}\asymp\frac1{\om^\star(z_j)^p},\quad |z_j|\to1^-,
    \end{split}
    \end{equation*}
by Lemma~\ref{Lemma:replacement-Lemmas-Memoirs}. This combined with \eqref{eq:chumino} and the equivalence
(i)$\Leftrightarrow$(iii), proved in \cite[Theorem~1]{PR2016/2}, gives the assertion.
\end{Prf}

\medskip

In view of Theorems~\ref{MTMSchatten} and~\ref{b1} it is
natural to ask whether or not the condition $\widetilde{\mathcal
T}_\mu\in L^p_{\om/\om^\star}$ characterizes the Schatten class Toeplitz operators for the whole class $\DD$ of doubling weights. The next result answers this question in negative.

\begin{proposition}\label{pr:noberI}
For each $1<p<\infty$ there exist $\om\in\DD$ and a positive Borel
measure $\mu$ on $\D$ such that $\widetilde{\mathcal T}_\mu\in
L^p_{\om/\om^\star}$ but $\mathcal
T_\mu\notin\mathcal{S}_p(A^2_\om)$.
\end{proposition}

\begin{proof}
Let $\om(z)=\left[(1-|z|)\left(\log\frac{e}{1-|z|}
\right)^\a\right]^{-1}$, where $\a>-1$, and let $d\mu(z)=v(z)\,dA(z)$, where
$v(z)=(1-|z|)^{-1+\frac{1}{p}}\left(\log\frac{e}{1-|z|}
\right)^{-\a+1-\b}$ and $0<\b<\frac{1}{p}$. Then
\eqref{Eq:BerezinTransformation}, Theorem~\ref{Lemma:kernels} and
Lemma~\ref{Lemma:replacement-Lemmas-Memoirs} yield
    $$
    \widetilde{\mathcal T}_\mu(z)\asymp \om^\star(z)\int_{\D}|B^\om_z(\z)|^2\,v(\z)\,dA(\z)
    \asymp\frac{\widehat{v}(z)}{\widehat{\om}(z)}\asymp(1-|z|)^{\frac{1}{p}}\left(\log\frac{e}{1-|z|}
    \right)^{-\b},\quad |z|\ge\frac12.
    $$
Therefore
    \begin{equation*}
    \begin{split}
    &\int_{\D\setminus D(0,\frac12)}\left(\widetilde{\mathcal T}_\mu(z)\right)^p\,\frac{\om(z)}{\om^\star(z)}\,dA(z)\\ &\asymp\int_{{\D\setminus D(0,\frac12)}}\left((1-|z|)^{\frac{1}{p}}\left(\log\frac{e}{1-|z|} \right)^{-\b}\right)^p\frac{dA(z)}{(1-|z|)^2\log\frac{e}{1-|z|}}\\ &=\int_{{\D\setminus D(0,\frac12)}}\frac{dA(z)}{(1-|z|)\left(\log\frac{e}{1-|z|} \right)^{\b p+1}}<\infty,
    \end{split}
    \end{equation*}
and thus $\widetilde{\mathcal T}_\mu\in L^p_{\om/\om^\star}$. However, for each $r\in(0,1)$,
    $$
    \ddot{\mu}_r(z)=\frac{\mu(\Delta(z,r))}{\om^\star(z)}\asymp \frac{(1-|z|)^2v(z)}{\om^\star(z)}
    \asymp (1-|z|)^{\frac{1}{p}}\left(\log\frac{e}{1-|z|}
    \right)^{-\b},\quad|z|\ge\frac12,
    $$
and hence
    \begin{equation*}
    \begin{split}
    \|\ddot{\mu}_r\|^p_{L^p\left(\frac{dA(z)}{(1-|z|)^2}\right)}
    &\gtrsim\int_{\D\setminus D(0,\frac12)}\left((1-|z|)^{\frac{1}{p}}\left(\log\frac{e}{1-|z|} \right)^{-\b}\right)^p\,\frac{dA(z)}{(1-|z|)^2}\\
    &=\int_{\D\setminus D(0,\frac12)}\frac{dA(z)}{(1-|z|)\left(\log\frac{e}{1-|z|} \right)^{p\b}}=\infty.
    \end{split}
    \end{equation*}
Consequently, $\mathcal T_\mu\notin\mathcal{S}_p(A^2_\om)$ by
Theorem~\ref{MTMSchatten}.
\end{proof}

The asymptotic relation $\om(z)/\om^\star(z)\asymp(1-|z|)^{-2}$, valid for each $\om\in\R$ and $z\in\D$ uniformly bounded away from the origin, has been repeatedly used in this paper. This relation fails for $\om\in\DD\setminus\R$ and, for example, the doubling weight $\om(z)=\left[(1-|z|)\left(\log\frac{e}{1-|z|}
\right)^\a\right]^{-1}$, where $\a>-1$, satisfies $\om(z)(1-|z|)^2/\om^\star(z)\asymp\left(\log\frac{e}{1-|z|}\right)^{-1}\to0$, as $|z|\to1^-$.
The last result of this section shows that this innocent looking difference is significant concerning the conditions
 $\widetilde{\mathcal{T}}_\mu\in L^1\left(\frac{dA}{(1-|\cdot|)^2}\right)$ and $\widetilde{\mathcal{T}}_\mu\in L^1_{\om/\om^\star}$.
 Therefore one may not replace $L^1_{\om/\om^\star}$ by $L^1\left(\frac{dA}{(1-|\cdot|)^2}\right)$ in the statement of Theorem~\ref{b1}.

\begin{proposition}
There exists $\om\in\DD$ and a positive Borel measure $\mu$ on $\D$
such that $\mathcal T_\mu\in\mathcal{S}_1(A^2_\om)$ and
$\widetilde{\mathcal{T}}_\mu\notin L^1\left(\frac{dA}{(1-|\cdot|)^2}\right)$.
\end{proposition}

\begin{proof}
Choose $\om(z)=\left[(1-|z|)\left(\log\frac{e}{1-|z|} \right)^\a\right]^{-1}$, where $\a>2$, and $d\mu(z)=u(z)\,dA(z)$, where
$u(z)=\left(\log\frac{e}{1-|z|} \right)^{-\b-\a}$ and $0<\b< \min\{1,\a-2\}$.
Then, by Lemma~\ref{Lemma:replacement-Lemmas-Memoirs},
    \begin{equation*}
    \begin{split}
    \|B^\om_z\|^2_{L^2_\mu}
    =\sum_{n=0}^\infty\frac{|z|^{2n}}{\left[\left(v_\a\right)_{2n+1}\right]^2}u_n
    \asymp\sum_{n=1}^\infty\frac{|z|^{2n}}{(n+1)}(\log n)^{\a-\b-2}\asymp \left(\log
    \frac{e}{1-|z|}\right)^{\a-\b-1},
    \end{split}
    \end{equation*}
and hence
    $$
    \widetilde{\mathcal T}_\mu(z)
    \asymp\om^\star(z)\|B^\om_z\|^2_{L^2_\mu}\asymp (1-|z|)\left(\log
    \frac{e}{1-|z|}\right)^{-\b},\quad|z|\ge \frac12,
    $$
by \eqref{Eq:BerezinTransformation}. It follows that $\widetilde{\mathcal{T}}_\mu\notin L^1\left(\frac{dA}{(1-|\cdot|)^2}\right)$. However,
    \begin{equation*}
    \begin{split}
    \int_{\D\setminus D(0,\frac12)}\widetilde{\mathcal T}_\mu(z)\frac{\om(z)}{\om^\star(z)}\,dA(z)
    &\asymp\int_{\D}\frac{dA(z)}{(1-|z|)\left(\log\frac{e}{1-|z|} \right)^{\b +1}}<\infty,
    \end{split}
    \end{equation*}
and hence $\mathcal T_\mu\in\mathcal{S}_1(A^2_\om)$ by Theorem~\ref{b1}.
\end{proof}

\section{Schatten class composition
operators}\label{sec:composition}

The main purpose of this section is to prove Theorem~\ref{Thm:CompositionBoundedValence}. The
following result of its own interest plays a role in the proof.

\begin{proposition}\label{Thm:CompositionSufficientNecessary}
Let $0<p<\infty$ and $\omega\in\DD$, and let $\vp$ be an analytic
self-map of $\D$. Then the condition \eqref{39intro}
is sufficient if $0<p\le2$ and necessary if $2\le p<\infty$ for
$C_\vp$ to belong to $\SSS_p(A^2_\omega)$.
\end{proposition}

\begin{proof}
First observe that
    \begin{equation}\label{eq:c2}
    \langle f,C_\vp^\star(b_z^\omega)\rangle_{A^2_\omega}=\langle
    C_\vp(f),b_z^\omega\rangle_{A^2_\omega}
    =\|B_z^\omega\|_{A^2_\omega}^{-1}\langle C_\vp(f),B_z^\omega\rangle_{A^2_\omega}
    =\|B_z^\omega\|_{A^2_\omega}^{-1}f(\vp(z)),
    \end{equation}
and hence
$C_\vp^\star(b_z^\omega)=\|B_z^\omega\|_{A^2_\omega}^{-1}B_{\vp(z)}^\omega$.
Consequently,
% taking $T=C_\vp^\star$ in \eqref{eq:c1}
    \begin{equation}\label{40}
    \|C_\vp^\star(b_z^\omega)\|_{A^2_\omega}^2
    =\frac{\|B_{\vp(z)}^\omega\|_{A^2_\omega}^2}{\|B_{z}^\omega\|_{A^2_\omega}^2}
    \asymp\frac{\om(S(z))}{\omega((S(\vp(z))))},\quad z\in\D,
    \end{equation}
by Theorem~\ref{Lemma:kernels}. This and
Lemma~\ref{Lemma:replacement-Lemmas-Memoirs} yield
    \begin{eqnarray*}
    \int_\D\left(\frac{\omega^\star(z)}{\omega^\star(\vp(z))}\right)^\frac{p}{2}\frac{\omega(z)}{\omega^\star(z)}\,dA(z)
    \asymp\int_\D\|C_\vp^\star(b_z^\omega)\|_{A^2_\omega}^p\frac{\omega(z)}{\omega^\star(z)}\,dA(z)
     =
    \int_\D\|\widetilde{T}(z)\|_{A^2_\omega}^{\frac{p}{2}}\frac{\omega(z)}{\omega^\star(z)}\,dA(z),
    \end{eqnarray*}
where $T=C_\vp C_\vp^\star$. The assertion follows from
\cite[Theorem~1.26]{Zhu} and Lemma~\ref{b3}.

An alternative way to establish the assertions is to follow the
reasoning in \cite[p.~1143]{LuZhu92}.
\end{proof}

\begin{Prf}{\em{Theorem~\ref{Thm:CompositionBoundedValence}.}}
Since $C_\vp^\star$ can be formally computed as
    \begin{equation*}
    \begin{split}
    C_\vp^\star(f)(z)
    &= \langle C_\vp^\star f,B_z^\omega\rangle_{A^2_\omega}=\langle
    f, C_\vp(B_z^\omega)\rangle_{A^2_\omega}
    =\langle f,B_z^\omega(\vp)\rangle_{A^2_\omega}\\
    &=\int_{\D}f(\z)B_z^\omega(\vp(\z))\om(\z)\,dA(\z),
    \end{split}
    \end{equation*}
it follows that
    $$
    C_\vp^\star C_\vp(f)(z)=\int_{\D}f(\vp(\z))B_z^\omega(\vp(\z))\om(\z)\,dA(\z).
    $$
Let $\mu$ be the pull-back measure defined by $\mu(E)=\om\left(\vp^{-1}(E)\right)$. Then
    $$
    C_\vp^\star C_\vp(f)(z)= \int_{\D}f(u)B_z^\omega(u)\,d\mu(u)=\Tm(f)(z),
    $$
and hence
%Recall that $C_\vp^\star C_\vp=T_\mu$, where
%$\mu=\mu_{2,1,\omega}$, and hence
$C_\vp\in\SSS_p(A^2_\omega)$ if and only if
$\mathcal{T}_\mu\in\SSS_{p/2}(A^2_\omega)$ by
\cite[Theorem~1.26]{Zhu}. Therefore, by Theorems~\ref{MTMSchatten}
and~\ref{Thm:CompositionSufficientNecessary}, it suffices to show
that \eqref{39intro} implies $\widetilde{\mathcal{T}}_\mu\in
L^\frac{p}{2}_{\om/\om^\star}$. To see this, we use
Theorem~\ref{Lemma:kernels} to write
    $$
    \widetilde{\mathcal{T}}_\mu(z)=\langle
    \mathcal{T}_\mu(b_z^\om),b_z^\om\rangle_{A^2_\omega}=\int_\D\frac{|B_z^\omega(\zeta)|^2}{\|B_z^\omega\|^2_{A^2_\om}}\,d\mu(\zeta)
    \asymp\omega(S(z))\int_\D|B_z^\omega(\vp(\zeta))|^2\omega(\zeta)\,dA(\zeta).
    $$
We will now argue as in \cite[p.~180]{Zhu2001}. Note first that
\cite[Theorem~4.2]{PelRat} gives
    $$
    \widetilde{\mathcal{T}}_\mu(z)\asymp\omega(S(z))|B_z^\omega(\vp(0))|^2
    +\omega(S(z))\int_\D|(B_z^\omega)'(\vp(\zeta))|^2|\vp'(\zeta)|^2\omega^\star(\zeta)\,dA(\zeta).
    $$
Hence it suffices to show that
    $$
    \Phi(z)=\omega(S(z))\int_\D|(B_z^\omega)'(\vp(\zeta))|^2|\vp'(\zeta)|^2\omega^\star(\zeta)\,dA(\zeta)
    $$
belongs to $L^{p/2}_{\omega/\omega^\star}$. To do this we will
use Shur's test with two measures \cite[Theorem~3.8]{Zhu}. Let
    $$
    \psi(\zeta)=\frac{\omega^\star(\zeta)}{\omega^\star(\vp(\zeta))},\quad
    d\nu(\zeta)=\frac{\omega(\vp(\zeta))}{\omega^\star(\vp(\zeta))}|\vp'(\zeta)|^2\,dA(\zeta)
    $$
and
    $$
    H(z,\zeta)=\frac{|(B_z^\omega)'(\vp(\zeta))|^2\omega(S(z))\omega^\star(\vp(\zeta))^2}{\omega(\vp(\zeta))},
    $$
so that the operator
    $$
    T(f)=\int_\D H(z,\zeta)f(\zeta)\,d\nu(\zeta)
    $$
satisfies $T(\psi)=\Phi$. Since $\vp$ is of bounded valence, we
obtain
    \begin{eqnarray*}
    \int_\D H(z,\zeta)\,d\nu(\zeta)
    &=&\omega(S(z))\int_\D|(B_z^\omega)'(\vp(\zeta))|^2\omega^\star(\vp(\zeta))|\vp'(\zeta)|^2\,dA(\zeta)\\
    &\asymp &\omega(S(z))\int_\D|(B_z^\omega)'(\xi)|^2\omega^\star(\xi)\,dA(\xi)\asymp1
    \end{eqnarray*}
by Theorem~\ref{Lemma:kernels}. Moreover, by
Theorem~\ref{Lemma:kernels},
    \begin{equation*}
    \begin{split}
    \int_\D
    H(z,\zeta)\frac{\omega(z)}{\omega^\star(z)}\,dA(z)
    &\asymp\frac{\omega^\star(\vp(\zeta))^2}{\omega(\vp(\zeta))}\int_\D|(B_z^\omega)'(\vp(\zeta))|^2\omega(z)\,dA(z)\\
    &=\frac{\omega^\star(\vp(\zeta))^2}{\omega(\vp(\zeta))}\int_\D|(B_{\vp(\zeta)}^\omega)'(z)|^2\omega(z)\,dA(z)
    \lesssim 1,
    \end{split}
    \end{equation*}
because $\omega\in\R$. Now $\psi\in L^{p/2}_{\omega/\omega^\star}$
by the assumption \eqref{39intro}, and
$\nu\lesssim\omega/\omega^\star$ by the Schwarz-Pick lemma and the
assumption $\omega\in\R$, so $\psi\in L^{p/2}_{\nu}$. Therefore we
may apply Schur's test (with both test functions equal to $1$) to
deduce that $T$ is a  bounded operator from $L^{p/2}_{\nu}$  into
$L^{p/2}_{\omega/\omega^\star}$, and thus, in particular,
$T(\psi)=\Phi\in L^{p/2}_{\omega/\omega^\star}$. Therefore
$\widetilde{\mathcal T}_\mu\in L^{p/2}_{\omega/\omega^\star}$ as
desired.
\end{Prf}

The following result is parallel to Proposition~\ref{Thm:CompositionSufficientNecessary}. By the Schwarz-Pick lemma, \eqref{39intro} implies \eqref{103} for all $0<p<\infty$ and $\om\in\R$, and therefore the case $0<p<2$ is of particular interest.

\begin{proposition}\label{Thm:Composition-Extra}
Let $0<p<\infty$ and $\omega\in\DD$, and let $\vp$ be an analytic
self-map of $\D$. Then the condition
    \begin{equation}\label{103}
    \int_\D\left(\frac{\omega^\star(z)}{\omega^\star(\vp(z))}\right)^\frac{p}{2}\frac{|\vp'(z)|^p(1-|z|^2)^{p-2}}{(1-|\vp(z)|^2)^p}\,dA(z)<\infty
    \end{equation}
is sufficient if $0<p\le2$ and necessary if $2\le p<\infty$ for
$C_\vp$ to belong to $\SSS_p(A^2_\omega)$.
\end{proposition}

\begin{proof}
Let first $p\ge2$. The Schwarz-Pick lemma, a change of variable and
a standard inequality yield
    \begin{equation*}
    \begin{split}
    &\int_\D\left(\frac{\omega^\star(z)}{\omega^\star(\vp(z))}\right)^\frac{p}{2}\frac{|\vp'(z)|^p(1-|z|^2)^{p-2}}{(1-|\vp(z)|^2)^p}\,dA(z)\\
    &\le\int_\D\left(\frac{\omega^\star(z)}{\omega^\star(\vp(z))}\right)^\frac{p}{2}\frac{|\vp'(z)|^2}{(1-|\vp(z)|^2)^2}\,dA(z)\\
    &=\int_\D\frac{N_{\varphi,(\om^\star)^{p/2}}(\z)}{\omega^\star(\z)^\frac{p}{2}}
    \frac{dA(\z)}{(1-|\z|^2)^2}
    \le\int_\D \left(\frac{N_{\varphi,\om^\star}(\z)}{\omega^\star(\z)}\right)^\frac{p}{2}
    \frac{dA(\z)}{(1-|\z|)^2},
    \end{split}
    \end{equation*}
and hence the assertion follows by \cite[Theorem~3]{PR2016/2}. A similar
reasoning shows the case $p<2$.
\end{proof}

\end{document}